\documentclass{amsart}

\usepackage{hyperref}

\newcommand*{\mailto}[1]{\href{mailto:#1}{\nolinkurl{#1}}}
\newcommand{\arxiv}[1]{\href{http://arxiv.org/abs/#1}{arXiv:#1}}
\newcommand{\doi}[1]{\href{http://dx.doi.org/#1}{doi: #1}}
\newcommand{\msc}[1]{\href{http://www.ams.org/msc/msc2010.html?t=&s=#1}{#1}}

\newtheorem{theorem}{Theorem}[section]

\newtheorem{lemma}[theorem]{Lemma}
\newtheorem{proposition}[theorem]{Proposition}
\newtheorem{corollary}[theorem]{Corollary}

\newtheorem{remark}[theorem]{Remark}
\newtheorem{example}{Example}[section]

\newcommand{\R}{{\mathbb R}}
\newcommand{\N}{{\mathbb N}}

\newcommand{\C}{{\mathbb C}}

\newcommand{\spr}[2]{\langle #1 , #2 \rangle}

\newcommand{\Qr}{\mathsf{q}}

\newcommand{\Vr}{\mathsf{v}}
\newcommand{\Wr}{\mathsf{w}}

\newcommand{\T}{\mathrm{T}}
\newcommand{\Sr}{\mathrm{s}}

\newcommand{\F}{\mathcal{F}}

\newcommand{\E}{\mathrm{e}}
\newcommand{\I}{\mathrm{i}}

\newcommand{\supp}{\mathrm{supp}}

\newcommand{\im}{\mathrm{Im}}
\newcommand{\re}{\mathrm{Re}}
\newcommand{\loc}{{\mathrm{loc}}}
\newcommand{\cc}{{\mathrm{c}}}

\newcommand{\OO}{\mathcal{O}}

\newcommand{\ledot}{\,\cdot\,}
\newcommand{\redot}{\cdot\,}

\newcommand{\NL}{(0-)}
\newcommand{\NLz}{(z,0-)}

\newcommand{\qd}{{[1]}}

\newcommand{\Hast}{\dot{H}^1[0,L)}
\newcommand{\Hasto}{\dot{H}^1_0[0,L)}
\newcommand{\cH}{\mathcal{H}}
\newcommand{\Lmu}{L^2(\R;\mu)}
\newcommand{\dip}{\upsilon}

\numberwithin{equation}{section}

\begin{document}

\title[Absolutely continuous spectrum]{On the absolutely continuous spectrum of~generalized~indefinite strings}
 
\author[J.\ Eckhardt]{Jonathan Eckhardt}
\address{Department of Mathematical Sciences\\ Loughborough University\\ Loughborough LE11 3TU \\ UK\\  and  Faculty of Mathematics\\ University of Vienna\\ Oskar-Morgenstern-Platz 1\\ 1090 Wien\\ Austria}
\email{\mailto{J.Eckhardt@lboro.ac.uk};\ \mailto{jonathan.eckhardt@univie.ac.at}}
\urladdr{\url{http://homepage.univie.ac.at/jonathan.eckhardt/}}

\author[A.\ Kostenko]{Aleksey Kostenko}
\address{Faculty of Mathematics and Physics\\ University of Ljubljana\\ Jadranska 19\\ 1000 Ljubljana\\ Slovenia\\ and Faculty of Mathematics\\ University of Vienna\\ Oskar-Morgenstern-Platz 1\\ 1090 Wien\\ Austria}
\email{\mailto{Aleksey.Kostenko@fmf.uni-lj.si};\ \mailto{Oleksiy.Kostenko@univie.ac.at}}
\urladdr{\url{http://www.mat.univie.ac.at/~kostenko/}}

\thanks{{\it Research supported by the Austrian Science Fund (FWF) under Grants No.~P29299 (J.E.) and P28807 (A.K.) as well as by the Slovenian Research Agency (ARRS) under Grant No.\ J1-9104 (A.K.)}}


\begin{abstract}
We investigate absolutely continuous spectrum of generalized indefinite strings. 
By following an approach of Deift and Killip, we establish stability of the absolutely continuous spectra of two model examples of generalized indefinite strings under rather wide perturbations. 
In particular, one of these results allows us to prove that the absolutely continuous spectrum of the isospectral problem associated with the conservative Camassa--Holm flow in the dispersive regime is essentially supported on the interval $[1/4,\infty)$.
\end{abstract}

\maketitle

\section{Introduction}

It is well known that the study of spectral types of self-adjoint operators in Hilbert spaces is crucial for understanding the corresponding quantum dynamics. 
For this reason, even a huge amount of effort has been put in understanding only one of the most basic models -- one-body Schr\"odinger operators with (in a certain sense) decaying potentials. 
The standard framework to investigate such Schr\"odinger operators is perturbation theory, where one considers a given operator as an additive (or maybe singular, that is, in the sense of resolvents) perturbation of another operator (for example, the free Hamiltonian) whose spectral properties are rather well understood. 
According to the Aronszajn--Donoghue theory, the absolutely continuous spectrum of a self-adjoint operator is the most stable part and theorems by Rosenblum--Kato and Birman--Krein provide stability of the absolutely continuous spectrum under trace class perturbations. 
It turns out that these results are sharp since by the Weyl--von Neumann--Kuroda theorem, absolutely continuous spectrum may be turned  into pure point spectrum by a perturbation of any Schatten class weaker than trace class. 

On the other hand, by restricting to particular classes of perturbations (for Schr\"o\-din\-ger operators, the considered perturbations are usually multiplication operators) one can hope for better results. 
For example, in their seminal work \cite{deki99}, Deift and Killip proved that the absolutely continuous spectrum of a one-dimensional Schr\"odinger operator 
\begin{align}
\label{eqnHV}
  -\frac{d^2}{dx^2} + V(x)
\end{align}
acting in $L^2(\R)$ coincides with $[0,\infty)$ as long as $V$ is a real-valued potential in $L^2(\R)$ (see also \cite{hs14,ki02,mnv01,ry05} for further results and references). 
A similar statement for multi-dimensional operators is known as {\em Simon's conjecture} \cite[Conjecture~20.2]{sim17} and still open, although this area has seen significant progress in the 1990s and 2000s.   
We are not attempting to give an overview of this topic (which would be a rather difficult task) but only point to the review articles \cite{dk05} and \cite{sim00, sim18, sim17}, where further references can be found. 

Our aim here is to investigate spectral types, or more precisely, the absolutely continuous part of the spectrum of {\em generalized indefinite strings}, introduced recently in \cite{IndefiniteString}. 
We recall briefly (see Section~\ref{sec:prelim} for further details) that a generalized indefinite string is a triple $(L,\omega,\dip)$ such that $L\in(0,\infty]$, $\omega$ is a real-valued distribution in $H^{-1}_{\loc}[0,L)$ and $\dip$ is a non-negative Borel measure on $[0,L)$. 
 Associated with such a triple is the ordinary differential equation of the form  
  \begin{align}\label{eqnDEIndStr}
  -f'' = z\, \omega f + z^2 \dip f
 \end{align}
 on $[0,L)$, where $z$ is a complex spectral parameter. 
 Spectral problems of this type are of interest for at least two reasons. 
Firstly, they constitute a canonical model for operators with simple spectrum; see \cite{IndefiniteString, IndMoment}. 
Secondly, they are of relevance in connection with certain completely integrable nonlinear wave equations  (most prominently, the Camassa--Holm equation \cite{caho93}), for which these kinds of spectral problems  arise as isospectral problems.  
 
 Results on spectral types, even for the special case of a Krein string
   \begin{align}\label{eqnClaS}
  -f'' = z\, \omega f,
 \end{align}
that is, when $\omega$ is a non-negative Borel measure and $\dip$ vanishes identically, are rather scarce (however, let us mention the recent paper by Bessonov and Denisov \cite{bd17}). 
 The main reason for this lies in the fact that the spectral parameter appears in the {\em wrong} place, which does not allow to view~\eqref{eqnClaS} as an additive perturbation immediately. 
 One of the standard approaches here is to transform the differential equation~\eqref{eqnClaS} into Schr\"odinger form by means of a Liouville transformation and then apply the well-developed spectral theory for one-dimensional Schr\"odinger operators. 
 However, this immediately requires strictly positive and sufficiently smooth weights $\omega$, assumptions which are too restrictive for our applications. 
 In this context, let us also mention that spectral theory of {\em indefinite strings}, that is, when the coefficient $\omega$ in \eqref{eqnClaS} is a real-valued Borel measure, is significantly more intricate when compared to the case of Krein strings (see \cite{fl96,ko13} for example).

Our approach to the absolutely continuous spectrum of generalized indefinite strings follows the elegant ideas of Deift and Killip \cite{deki99}. 
In order to implement their approach, we need two main ingredients. 
The first ingredient is a continuity property for the correspondence between generalized indefinite strings and their associated Weyl--Titchmarsh functions  (see Section \ref{sec:prelim} for more details) obtained in \cite[Proposition~6.2]{IndefiniteString}. 
The second ingredient is a so-called {\em dispersion relation} or {\em trace formula}, which provides a relation between the spectral/scattering data and the coefficients in the differential equation, and hence allows to {\em control} the spectral measure by means of the coefficients. 
For Schr\"odinger operators \eqref{eqnHV}, such relations have been discovered by Faddeev and Zakharov \cite{faza71} (see also \cite{bf60}) in connection with the Korteweg--de Vries equation, where they give rise to conserved quantities. 
Generalized indefinite strings on the other side are of the same importance for the conservative Camassa--Holm flow \cite{chlizh06, coiv08, ConservCH, LagrangeCH, ConservMP, hoiv11} as the one-dimensional Schr\"odinger operator~\eqref{eqnHV} for the Korteweg--de Vries equation.
In particular, dispersion relations for~\eqref{eqnDEIndStr} are of special interest in this connection because they give rise to conserved quantities of the flow. 
Under the rather restrictive assumptions that $\dip$ vanishes identically and $\omega$ is a sufficiently smooth and positive function, such relations have been established before (see, for example, \cite{coiv06}). 
Removing the positivity assumption however immediately creates an (in general) infinite number of negative eigenvalues for \eqref{eqnDEIndStr} and drastically complicates the situation. 
For this reason, the derivation of corresponding dispersion relations for generalized indefinite strings will require more efforts (see Lemma \ref{lemTF0} and Lemma \ref{lemTF2}).

Our main result in regards to the conservative Camassa--Holm flow can be deduced readily from Theorem~\ref{thmApp} by standard arguments:

\begin{theorem}\label{thmCHac}
Let $u$ be a real-valued function on $\R$ such that $u-1$ belongs to $H^1(\R)$ and let $\dip$ be a non-negative finite Borel measure on $\R$.  
Then the essential spectrum of the spectral problem
\begin{align}\label{eq:CHspec}
 - g'' + \frac{1}{4}g & = z\,\omega\, g + z^2 \dip\, g,  & \omega & = u - u'',
\end{align}
 coincides with the interval $[1/4,\infty)$ and the absolutely continuous spectrum is of multiplicity two and essentially supported on $[1/4,\infty)$. 
\end{theorem}  

\noindent 
In the special case when $\omega-1$ is a measure with strong enough decay and $\dip$ vanishes identically, this result has been established before by Bennewitz, Brown and Weikard \cite{bebrwe08, bebrwe12}. 
Under these additional assumptions the essential spectrum is in fact purely absolutely continuous, which will not continue to hold for our set of coefficients. 
Theorem~\ref{thmCHac} is relevant because it covers the entire natural phase space for the conservative Camassa--Holm flow in the dispersive regime (see \cite{brco07,grhora12, hora07}), which comprises the Camassa--Holm equation \cite{co01} as well  as the two-component Camassa--Holm system \cite{chlizh06, coiv08, hoiv11}. 
Note here that the additional coefficient $\dip$ is related to the energy measure $\mu$ from \cite{brco07,grhora12, hora07} in a simply way, essentially via 
\begin{align}\label{eq:CHmu}
  \mu(B) =  \dip(B) + \int_B  (u(x)-1)^2 + u'(x)^2 dx
\end{align}
for every Borel set $B\subseteq\R$. 
The function $\rho$ that appears in the formulation of the two-component Camassa--Holm system is simply the square root of the Radon--Nikodym derivative of the absolutely continuous part of the measure $\dip$. 

In conclusion, let us sketch the content of the article. 
Section~\ref{sec:prelim} is of preliminary character and collects necessary notions and facts from the spectral theory of generalized indefinite strings. 
Section \ref{secMR} contains the statements of our main results (Theorem~\ref{thm0} and Theorem~\ref{thmalpha}) about the absolutely continuous spectrum for certain classes of generalized indefinite strings, which are rather strong perturbations of explicitly solvable models (Example~\ref{exa0} and Example~\ref{exaalpha}).
Although we will not present the necessary details here, these perturbations can indeed be interpreted in a certain way as additive perturbations, which are only of Hilbert--Schmidt class in general however. 
Our main theorems seem to be new even for the special case of Krein strings and will be proved in Sections \ref{secPr1} and \ref{secPr2} respectively, after deriving some useful auxiliary facts in Section~\ref{secES}. 
 Even though we will generally follow the simple method by Deift and Killip from \cite{deki99}, the necessary ingredients are not as readily available for our spectral problem and need to be established first. 
In Section~\ref{secAPP}, we will then show how corresponding results for the spectral problem~\eqref{eq:CHspec} can be derived readily from Theorem~\ref{thmalpha}. 
The final section provides another application of our results to one-dimensional Hamiltonians with $\delta'$-interactions.

\section{Generalized indefinite strings}\label{sec:prelim} 

 We will first introduce several spaces of functions and distributions.  
 For every fixed $L\in(0,\infty]$, we denote with $H^1_{\loc}[0,L)$, $H^1[0,L)$ and $H^1_{\cc}[0,L)$ the usual Sobolev spaces. 
 To be precise, this means  
\begin{align}
H^1_{\loc}[0,L) & =  \lbrace f\in AC_{\loc}[0,L) \,|\, f'\in L^2_{\loc}[0,L) \rbrace, \\
 H^1[0,L) & = \lbrace f\in H^1_{\loc}[0,L) \,|\, f,\, f'\in L^2[0,L) \rbrace, \\ 
 H^1_{\cc}[0,L) & = \lbrace f\in H^1[0,L) \,|\, \supp(f) \text{ compact in } [0,L) \rbrace.
\end{align}
The space of distributions $H^{-1}_{\loc}[0,L)$ is the topological dual of $H^1_{\cc}[0,L)$. 
One notes that the mapping $\Qr\mapsto\chi$, defined by
 \begin{align}
    \chi(h) = - \int_0^L \Qr(x)h'(x)dx, \quad h\in H^1_{\cc}[0,L),
 \end{align} 
 establishes a one-to-one correspondence between $L^2_{\loc}[0,L)$ and $H^{-1}_{\loc}[0,L)$. 
The unique function $\Qr\in L^2_{\loc}[0,L)$ corresponding to some distribution $\chi\in H^{-1}_{\loc}[0,L)$ in this way will be referred to as {\em the normalized anti-derivative} of $\chi$.
 We say that a distribution in $H^{-1}_{\loc}[0,L)$ is real-valued if its normalized anti-derivative is real-valued almost everywhere on $[0,L)$. 

A particular kind of distribution in $H^{-1}_{\loc}[0,L)$ arises from Borel measures on the interval $[0,L)$.
 In fact, if $\chi$ is a complex-valued Borel measure on $[0,L)$, then we will identify it with the distribution in $H^{-1}_{\loc}[0,L)$ given by  
 \begin{align}
  h \mapsto \int_{[0,L)} h\,d\chi. 
 \end{align}
The normalized anti-derivative $\Qr$ of such a $\chi$ is simply given by the left-continuous distribution function 
 \begin{align}
 \Qr(x)=\int_{[0,x)}d\chi
 \end{align}
 for almost all $x\in [0,L)$, as an integration by parts (use, for example, \cite[Exercise~5.8.112]{bo07}, \cite[Theorem~21.67]{hest65}) shows.  

 In order to obtain a self-adjoint realization of the differential equation~\eqref{eqnDEIndStr} in a suitable Hilbert space later, we also introduce the function space  
\begin{align}
\Hast & = \begin{cases} \lbrace f\in H^1_{\loc}[0,L) \,|\, f'\in L^2[0,L),~ \lim_{x\rightarrow L} f(x) = 0 \rbrace, & L<\infty, \\ \lbrace f\in H^1_{\loc}[0,L) \,|\, f'\in L^2[0,L) \rbrace, & L=\infty, \end{cases} 
\end{align}
 as well as the linear subspace  
\begin{align}
 \Hasto & = \lbrace f\in \Hast \,|\, f(0) = 0 \rbrace, 
\end{align}
 which turns into a Hilbert space when endowed with the scalar product 
 \begin{align}\label{eq:normti}
 \spr{f}{g}_{\Hasto} = \int_0^L f'(x) g'(x)^\ast dx, \quad f,\, g\in\Hasto.
 \end{align}
Here and henceforth, we will use a star to denote complex conjugation. 
The space $\Hasto$ can be viewed as a completion with respect to the norm induced by~\eqref{eq:normti} of the space of all smooth functions which have compact support in $(0,L)$. 
In particular, the space $\Hasto$ coincides algebraically and topologically with the usual Sobolev space $H^1_0[0,L)$ when $L$ is finite. 

 A generalized indefinite string is a triple $(L,\omega,\dip)$ such that $L\in(0,\infty]$, $\omega$ is a real-valued distribution in $H^{-1}_{\loc}[0,L)$ and $\dip$ is a non-negative Borel measure on the interval $[0,L)$.  
 Associated with such a generalized indefinite string is the inhomogeneous differential equation
 \begin{align}\label{eqnDEinho}
  -f''  = z\, \omega f + z^2 \dip f + \chi, 
 \end{align}
 where $\chi$ is a distribution in $H^{-1}_{\loc}[0,L)$ and $z$ is a complex spectral parameter. 
 Of course, this differential equation has to be understood in a weak sense:   
  A solution of~\eqref{eqnDEinho} is a function $f\in H^1_{\loc}[0,L)$ such that 
 \begin{align}
  \Delta_f h(0) + \int_{0}^L f'(x) h'(x) dx = z\, \omega(fh) + z^2 \dip(fh) +  \chi(h), \quad h\in H^1_{\cc}[0,L),
 \end{align}
 for some constant $\Delta_f\in\C$. 
 In this case, the constant $\Delta_f$ is uniquely determined and will always be denoted with $f'\NL$ for apparent reasons. 
 With this notion of solution, common basic existence and uniqueness results for the inhomogeneous differential equation~\eqref{eqnDEinho} are available; see \cite[Section~3]{IndefiniteString}. 

 The differential equation~\eqref{eqnDEinho} gives rise to a self-adjoint linear relation in a suitable Hilbert space. 
 In order to introduce this object, we consider the space 
 \begin{align}
 \cH = \Hasto\times L^2([0,L);\dip),
\end{align}
which turns into a Hilbert space when endowed with the scalar product
\begin{align}
 \spr{f}{g}_{\cH} = \int_0^L f_1'(x) g_1'(x)^\ast dx + \int_{[0,L)} f_2(x) g_2(x)^\ast d\dip(x), \quad f,\, g\in \cH.
\end{align} 
 The respective components of some vector $f\in\cH$ are hereby always denoted by adding subscripts, that is, with $f_1$ and $f_2$.  
Now the linear relation $\T$ in the Hilbert space $\cH$ is defined by saying that some pair $(f,g)\in\cH\times\cH$ belongs to $\T$ if and only if  the two equations 
\begin{align}\label{eqnDEre1}
-f_1'' & =\omega g_{1} + \dip g_{2}, &  \dip f_2 & =\dip g_{1},
\end{align}
hold. 
In order to be precise, the right-hand side of the first equation in~\eqref{eqnDEre1} has to be understood as the $H^{-1}_{\loc}[0,L)$ distribution given by 
\begin{align}
 h \mapsto \omega(g_1h) + \int_{[0,L)} g_2 h\, d\dip. 
\end{align} 
 Moreover, the second equation in~\eqref{eqnDEre1} holds if and only if $f_2$ is equal to $g_1$ almost everywhere on $[0,L)$ with respect to the measure $\dip$. 
  The linear relation $\T$ turns out to be self-adjoint in the Hilbert space $\cH$; see \cite[Theorem~4.1]{IndefiniteString}.
 
   A central object in the spectral theory for the linear relation $\T$ is the associated {\em Weyl--Titchmarsh function} $m$. 
   This function can be defined on $\C\backslash\R$ by 
 \begin{align}\label{eqnmdef}
  m(z) =  \frac{\psi'\NLz}{z\psi(z,0)},\quad z\in\C\backslash\R,
 \end{align} 
 where $\psi(z,\redot)$ is the unique (up to a constant multiples) non-trivial solution of the homogeneous differential equation
  \begin{align}\label{eqnDEho}
  -f'' = z\, \omega f + z^2 \dip f
 \end{align}
 which lies in $\Hast$ and $L^2([0,L);\dip)$, guaranteed to exist by \cite[Lemma~4.2]{IndefiniteString}. 
 It has been shown in \cite[Lemma~5.1]{IndefiniteString} that the Weyl--Titchmarsh function $m$ is a Herglotz--Nevanlinna function, that is, it is analytic, maps the upper complex half-plane $\C_+$ into the closure of the upper complex half-plane and satisfies the symmetry relation
 \begin{align}\label{eqnHNsym}
  m(z)^\ast = m(z^\ast), \quad z\in\C\backslash\R. 
 \end{align}
For this reason, the Weyl--Titchmarsh function $m$ admits an integral representation, which takes the form 
\begin{align}\label{eqnWTmIntRep}
 m(z) = c_1 z + c_2 - \frac{1}{Lz} +  \int_\R \frac{1}{\lambda-z} - \frac{\lambda}{1+\lambda^2}\, d\mu(\lambda), \quad z\in\C\backslash\R, 
\end{align}
for some constants $c_1$, $c_2\in\R$ with $c_1\geq0$ and a non-negative Borel measure $\mu$ on $\R$ with $\mu(\lbrace0\rbrace)=0$ for which the integral  
\begin{align}
 \int_\R \frac{d\mu(\lambda)}{1+\lambda^2} 
\end{align}
is finite. 
Here we employ the convention that whenever an $L$ appears in a denominator, the corresponding fraction has to be interpreted as zero if $L$ is infinite.

The measure $\mu$ turns out to be a {\em spectral measure} for the linear relation $\T$ in the sense that the operator part of $\T$ is unitarily equivalent to multiplication with the independent variable in $\Lmu$; see \cite[Theorem~5.8]{IndefiniteString}.
 Of course, this establishes an immediate connection between the spectral properties of the linear relation $\T$ and the measure $\mu$. 
 For example, the spectrum of $\T$ coincides with the topological support of $\mu$ and thus can be read off the singularities of $m$ (more precisely, the function $m$ admits an analytic continuation away from the spectrum of $\T$). 
 
 For the sake of simplicity, we shall always mean the spectrum of the corresponding linear relation when we speak of the spectrum of a generalized indefinite string in the following.
 The same convention applies to the various spectral types.

 \section{Absolutely continuous spectrum}\label{secMR}

  In general, any kind of (simple) spectrum can arise from a generalized indefinite string; see \cite[Theorem~6.1]{IndefiniteString}.
  Here, we are interested in the absolutely continuous spectrum of a particular class of generalized indefinite strings, which are suitable perturbations of the following explicitly solvable case. 
  
  \begin{example}\label{exa0}
   Let $S_0$ be the generalized indefinite string $(L_0,\omega_0,\dip_0)$ such that $L_0$ is infinite, the distribution $\omega_0$ is given via its normalized anti-derivative $\Wr_0$ by 
   \begin{align}
     \Wr_0(x) = x, \quad x\in[0,\infty),
   \end{align}
       and the measure $\dip_0$ vanishes identically.
       We note that under these assumptions, the corresponding differential equation~\eqref{eqnDEho} simply reduces to 
       \begin{align}
       - f'' = zf.
       \end{align}
       For every $z\in\C\backslash[0,\infty)$, the function $\psi_0(z,\redot)$  given by\footnote{In the following, we will always take the branch of the square root $\sqrt{\cdot}$ with cut along the positive semi-axis $[0,\infty)$ defined by $\sqrt{z}=\sqrt{|z|}\E^{\I \arg(z)/2}$ with $\arg(z)\in [0,2\pi)$.}
    \begin{align}
      \psi_0(z,x) = \E^{\I\sqrt{z}x}, \quad x\in[0,\infty),
    \end{align}   
    is a solution of this differential equation which lies in $\dot{H}^1[0,\infty)$.
   Consequently,  the corresponding Weyl--Titchmarsh function $m_0$ is given explicitly by 
    \begin{align}
      m_0(z) = \frac{\psi_0'\NLz}{z\psi_0(z,0)} = \frac{\I}{\sqrt{z}}, \quad z\in\C\backslash\R.
    \end{align}
    This guarantees that the spectrum of $S_0$ is purely absolutely continuous and coincides with the interval $[0,\infty)$. 
  \end{example}
  
 In particular, the essential spectrum of $S_0$ coincides with the interval $[0,\infty)$ and the absolutely continuous spectrum of $S_0$ is essentially supported on $[0,\infty)$.
 The latter means that every subset of $[0,\infty)$ with positive Lebesgue measure has positive measure with respect to the corresponding spectral measure $\mu_0$.
  It turns out that these two properties continue to hold under a rather wide class of perturbations. 
  
    \begin{theorem}\label{thm0}
    Let $S$ be a generalized indefinite string $(L,\omega,\dip)$ such that  $L$ is infinite and 
    \begin{align}\label{eqnCondS0}
      \int_0^\infty \left| \Wr(x) - c - \eta x\right|^2 dx  +  \int_{[0,\infty)}  d\dip& < \infty
    \end{align}
    for a real constant $c$ and a positive constant $\eta$, where $\Wr$ is the normalized anti-derivative of $\omega$.
    Then the essential spectrum of $S$ coincides with the interval $[0,\infty)$ and the absolutely continuous spectrum of $S$ is essentially supported on $[0,\infty)$. 
    \end{theorem}
    
    A proof for this result will be given in Section~\ref{secPr1}. 
        In view of the applications we have in mind (see Section~\ref{secAPP}), we are furthermore interested in perturbations of another explicitly solvable case involving a positive parameter $\alpha$.
  
  \begin{example}\label{exaalpha}
    Let $S_\alpha$ be the generalized indefinite string $(L_\alpha,\omega_\alpha,\dip_\alpha)$ such that $L_\alpha$ is infinite, the distribution $\omega_\alpha$ is given via its normalized anti-derivative $\Wr_\alpha$ by   
    \begin{align}
      \Wr_\alpha(x) = \frac{x}{1+2\sqrt{\alpha}x}, \quad x\in[0,\infty),
    \end{align} 
    and the measure $\dip_\alpha$ vanishes identically. 
    We note that under these assumptions, the corresponding differential equation~\eqref{eqnDEho} simply reduces to 
    \begin{align}\label{eq:SP2}
    - f''(x) = \frac{z}{(1+2\sqrt{\alpha}x)^2}  f(x),\quad x\in[0,\infty).
    \end{align}
    For every $z\in\C\backslash[\alpha,\infty)$, the function $\psi_\alpha(z,\redot)$ given by 
    \begin{align}\label{eq:psiC}
      \psi_\alpha(z,x) = (1+2\sqrt{\alpha}x)^{\I\frac{\sqrt{z-\alpha}}{2\sqrt{\alpha}} + \frac{1}{2}}, \quad x\in[0,\infty),
    \end{align}   
    is a solution of this differential equation which lies in $\dot{H}^1[0,\infty)$.
    Consequently, the corresponding Weyl--Titchmarsh function $m_\alpha$ is given explicitly by  
    \begin{align}
      m_\alpha(z) = \frac{\psi_\alpha'\NLz}{z\psi_\alpha(z,0)} = \frac{\I}{\sqrt{z-\alpha}+\I\sqrt{\alpha}}, \quad z\in\C\backslash\R.
    \end{align}
    This guarantees that the spectrum of $S_\alpha$ is purely absolutely continuous and coincides with the interval $[\alpha,\infty)$. 
  \end{example}
  
  In particular, the essential spectrum of $S_\alpha$ coincides with the interval $[\alpha,\infty)$ and the absolutely continuous spectrum of $S_\alpha$ is essentially supported on $[\alpha,\infty)$.
  These two properties are again preserved under a rather wide class of perturbations. 
  
  \begin{theorem}\label{thmalpha}
    Let $S$ be a generalized indefinite string $(L,\omega,\dip)$ such that  $L$ is infinite and 
    \begin{align}\label{eqnCondSalpha}
      \int_0^\infty   \Bigl| \Wr(x) - c - \frac{\eta x}{1+2\sqrt{\alpha}x}\Bigr|^2 x\, dx + \int_{[0,\infty)}  x\, d\dip(x) & < \infty
    \end{align}
    for a real constant $c$ and positive constants $\alpha$ and $\eta$, where $\Wr$ is the normalized anti-derivative of $\omega$.
    Then the essential spectrum of $S$ coincides with the interval $[\alpha/\eta,\infty)$ and the absolutely continuous spectrum of $S$ is essentially supported on $[\alpha/\eta,\infty)$. 
  \end{theorem}
  
   Although the proof of this result is quite similar to the one for Theorem~\ref{thm0} in principle, it will be carried out separately in Section~\ref{secPr2} due to differences in details.   
  
     \begin{remark}
   When the constant $\eta$ in the assumptions of Theorem~\ref{thm0} and Theorem~\ref{thmalpha} is negative, then the resulting spectral picture is simply reflected across the imaginary axis.  
     In this case, the essential spectrum of  $S$ coincides with the interval $(-\infty,0]$ and  $(-\infty,\alpha/\eta]$, respectively, and the absolutely continuous spectrum of $S$ is essentially supported on $(-\infty,0]$ and $(-\infty,\alpha/\eta]$, respectively. 
    \end{remark}
    
  As a conclusion to this section, let us mention that although we restricted to generalized indefinite strings on infinite intervals here, one can also consider perturbations of similar explicitly solvable cases on finite intervals.

 \section{Auxiliary facts about fundamental systems}\label{secES}

 Let the triple $(L,\omega,\dip)$ be an arbitrary generalized indefinite string and denote with $\Wr$ the normalized anti-derivative of $\omega$. 
 For every $z\in\C$, we introduce the fundamental system of solutions $\theta(z,\redot)$, $\phi(z,\redot)$ of the differential equation~\eqref{eqnDEho} satisfying the initial conditions
 \begin{align}
  \theta(z,0)& = \phi'\NLz =1, &  \theta'\NLz & = \phi(z,0) =0.
 \end{align}
 As the derivatives of these functions are only locally square integrable in general, we introduce the left-continuous {\em quasi-derivatives} $\theta^\qd(z,\redot)$, $\phi^\qd(z,\redot)$ on $[0,L)$ by 
 \begin{align}
 \label{eqnQDtheta} \theta^\qd(z,x) & = \theta'\NLz + z\int_0^x \Wr(t)\theta'(z,t)dt - z^2 \int_{[0,x)} \theta(z,t)d\dip(t), \\
  \label{eqnQDphi} \phi^\qd(z,x) & = \phi'\NLz + z\int_0^x \Wr(t)\phi'(z,t)dt - z^2 \int_{[0,x)} \phi(z,t)d\dip(t),
  \end{align}
 for $x\in[0,L)$, such that    
 \begin{align}\label{eqnQD}
   \theta^\qd(z,x) & = \theta'(z,x) + z\Wr(x)\theta(z,x), & \phi^\qd(z,x) & = \phi'(z,x) + z\Wr(x)\phi(z,x),
 \end{align}
 for almost all $x\in[0,L)$;  see \cite[Equation~(4.12)]{IndefiniteString}.
  It follows from \cite[Corollary~3.5]{IndefiniteString} that the functions 
   \begin{align}
   z & \mapsto \theta(z,x), & z & \mapsto \theta^\qd(z,x), & z & \mapsto \phi(z,x), & z & \mapsto \phi^\qd(z,x), 
 \end{align}
 are real entire for every fixed $x\in[0,L)$.
  At the origin, when $z$ is zero, one readily infers that our fundamental system is given explicitly by
 \begin{align}\label{eqncsatzero}
   \theta(0,x) & = 1, & \theta^\qd(0,x) & = 0, & \phi(0,x) & = x, & \phi^\qd(0,x) & = 1, 
 \end{align}
 for all $x\in[0,L)$. 
 We now seek to determine the derivatives of these functions with respect to the spectral parameter at the origin. 
 In order to state the following result, let us note that differentiation with respect to the spectral parameter will be denoted with a dot and is always meant to be done after taking quasi-derivatives.
 
 \begin{proposition}\label{propFS}
  For every $x\in[0,L)$, we have   
  \begin{align}
    \label{eqndtheta0} \dot{\theta}(0,x) & =  - \int_0^x \Wr(t)dt, &   \dot{\theta}^\qd(0,x) & =  0, \\
    \label{eqndphi0} \dot{\phi}(0,x) & = \int_0^x \int_0^t \Wr(s)ds\, dt - \int_0^x \Wr(t) t\, dt, &   \dot{\phi}^\qd(0,x) & = \int_0^x \Wr(t)dt,
  \end{align}
  as well as  
  \begin{align}
    \label{eqnddtheta0}  \ddot{\theta}(0,x) & = \biggl(\int_0^x \Wr(t)dt\biggr)^2 - 2 \int_0^x \int_0^t \Wr(s)^2 ds\,dt - 2 \int_0^x \int_{[0,t)} d\dip\,dt,                                                   \\
    \label{eqnddthetap0} \ddot{\theta}^\qd(0,x) & = -2 \int_0^x \Wr(t)^2 dt -2 \int_{[0,x)} d\dip,                                                    \\
    \label{eqnddphip0} \ddot{\phi}^\qd(0,x) & = \biggl(\int_0^x \Wr(t)dt\biggr)^2 - 2 \int_0^x \Wr(t)^2 t\, dt -2 \int_{[0,x)} t\, d\dip(t).                                                 
  \end{align}
 \end{proposition}

 \begin{proof}
  We first note that the second equality in~\eqref{eqndtheta0} follows from~\eqref{eqnQDtheta} since the first integral there is zero when $z$ is equal to zero. 
 Now consider the matrix function
  \begin{align*}
    Y(z,x) = \begin{pmatrix} \theta(z,x) & -z\phi(z,x) \\ -z^{-1}\theta^\qd(z,x) & \phi^\qd(z,x) \end{pmatrix}, \quad  z\in\C,~x\in[0,L),
  \end{align*}
  which is well-defined since the function $\theta^\qd(\ledot,x)$ has a root at zero. 
  It is an immediate consequence of \eqref{eqnQDtheta}, \eqref{eqnQDphi} and~\eqref{eqnQD}  that $Y$ satisfies the integral equation 
  \begin{align*}
   Y(z,x) = \begin{pmatrix} 1 & 0 \\ 0 & 1 \end{pmatrix} & + z \int_0^x  \begin{pmatrix} - \Wr(t) & -1 \\ \Wr(t)^2 & \Wr(t) \end{pmatrix} Y(z,t) dt \\
                                                                                     & + z \int_{[0,x)}  \begin{pmatrix} 0 & 0 \\ 1 & 0 \end{pmatrix}Y(z,t) d\dip(t), \quad x\in[0,L),~z\in\C.
  \end{align*}    
  For fixed $x\in[0,L)$, differentiating with respect to $z$ gives 
  \begin{align}\begin{split}\label{eqnIPYIEd}
   \dot{Y}(z,x) & =  \int_0^x  \begin{pmatrix} - \Wr(t) & -1 \\ \Wr(t)^2 & \Wr(t) \end{pmatrix} Y(z,t) dt + z \int_0^x  \begin{pmatrix} - \Wr(t) & -1 \\ \Wr(t)^2 & \Wr(t) \end{pmatrix} \dot{Y}(z,t) dt \\
                                                                                     & \qquad + \int_{[0,x)}  \begin{pmatrix} 0 & 0 \\ 1 & 0 \end{pmatrix}Y(z,t) d\dip(t) + z \int_{[0,x)}  \begin{pmatrix} 0 & 0 \\ 1 & 0 \end{pmatrix} \dot{Y}(z,t) d\dip(t), \quad z\in\C.
  \end{split}\end{align}  
  Evaluating at zero, we end up with 
  \begin{align*}
    \dot{Y}(0,x) =  \int_0^x  \begin{pmatrix} - \Wr(t) & -1 \\ \Wr(t)^2 & \Wr(t) \end{pmatrix} dt + \int_{[0,x)}  \begin{pmatrix} 0 & 0 \\ 1 & 0 \end{pmatrix} d\dip(t),
  \end{align*}
  which yields the first equality in~\eqref{eqndtheta0}, the second equality in~\eqref{eqndphi0} as well as~\eqref{eqnddthetap0}. 
  Differentiating~\eqref{eqnIPYIEd} once more and evaluating at zero, we obtain  
    \begin{align*}
   \ddot{Y}(0,x)  =  2 \int_0^x  \begin{pmatrix} - \Wr(t) & -1 \\ \Wr(t)^2 & \Wr(t) \end{pmatrix} \dot{Y}(0,t) dt   + 2 \int_{[0,x)}  \begin{pmatrix} 0 & 0 \\ 1 & 0 \end{pmatrix} \dot{Y}(0,t) d\dip(t),
  \end{align*}   
  which yields the first equality in~\eqref{eqndphi0}, \eqref{eqnddtheta0} as well as~\eqref{eqnddphip0} after performing some integrations by parts. 
 \end{proof}

 Although one could also compute the second derivative of $\phi(\ledot,x)$ at zero, we omitted to include it because the expression is somewhat lengthy and will not be needed in what follows. 
 In fact, to this end one just needs to note that 
  \begin{align}
    \ddot{\phi}(z,x) = \int_0^x \ddot{\phi}^\qd(z,t) dt - 2 \int_0^x \Wr(t) \dot{\phi}(z,t) dt - z \int_0^x \Wr(t) \ddot{\phi}(z,t)dt, \quad z\in\C,
  \end{align} 
  and plug in the expressions from Proposition~\ref{propFS} upon evaluating at zero.

 \section{Proof of Theorem~\ref{thm0}}\label{secPr1}

  To begin with, let us consider a particular class of generalized indefinite strings. 
  We assume that $(L,\omega,\dip)$ is a generalized indefinite string such that $L$ is infinite, there is an $R>0$ such that the normalized anti-derivative $\Wr$ of $\omega$ satisfies 
  \begin{align}
    \Wr(x) = x   
  \end{align}
  for almost all $x$ in $[R,\infty)$ and the measure $\dip$ vanishes on $[R,\infty)$. 
  In addition, let us also suppose that $\Wr$ is equal to a piecewise constant function almost everywhere on the interval $[0,R]$ and that the support of the measure $\dip$ is a finite set.
  The set of generalized indefinite strings with these properties will be denoted by $\F_0$. 
  Under our assumptions, for every $k$ in the upper complex half-plane $\C_+$, there is a {\em Jost solution} $f(k,\redot)$  of the differential equation~\eqref{eqnDEho} with $z=k^2 \in\C\backslash[0,\infty)$  
  such that 
  \begin{align}
    f(k,x) =  \E^{\I k x}, \quad x\in[R,\infty). 
  \end{align}   
  We note that since the function $f(k,\redot)$ clearly lies in $\dot{H}^1[0,\infty)$ and $L^2([0,\infty);\dip)$, the corresponding Weyl--Titchmarsh function $m$ is given by 
  \begin{align}\label{eqnmJost}
    m(k^2) = \frac{f'(k,0-)}{k^2f(k,0)}
  \end{align}
  as long as $k^2\in\C\backslash\R$. 
   Furthermore, we define the function $a$ on $\C_+$  via 
   \begin{align}\label{eq:a_0}
     a(k) = \frac{\I k f(k,0)+f'(k,0-)}{2\I k}, \quad k\in\C_+,
   \end{align}
  which can be viewed as the reciprocal transmission coefficient when the differential equation is suitably extended to the full line.

\begin{lemma}\label{lem:wronski}
The function $a$ has a unique continuation (denoted with $a$ as well for simplicity) to an entire function.
\end{lemma}   

\begin{proof}
If $\theta$, $\phi$ denotes the fundamental system of solutions of the differential equation~\eqref{eqnDEho} as in Section~\ref{secES}, then we may write
\begin{align*}
  f(k,x) = f(k,0)\theta(k^2,x) + f'(k,0-)\phi(k^2,x), \quad x\in[0,\infty),~k\in\C_+.
\end{align*}
Upon evaluating this function and its derivative at the point $R$, we get  
\begin{align*}
  \E^{\I k R} & = f(k,0)\theta(k^2,R) + f'(k,0-)\phi(k^2,R), \\
  \I k \E^{\I k R} + k^2 R \E^{\I k R} & = f(k,0)\theta^\qd(k^2,R) + f'(k,0-)\phi^\qd(k^2,R),
\end{align*}
which we can solve for $f(k,0)$ and $f'(k,0-)$ to obtain
\begin{align}
  \label{eqnJosttp} f(k,0) & =  \E^{\I k R}\phi^\qd(k^2,R) - (\I k + k^2 R)\E^{\I k R} \phi(k^2,R), \\
  \label{eqnJostdtp} f'(k,0-) & = -\E^{\I k R}\theta^\qd(k^2,R) + (\I k + k^2 R)\E^{\I k R}\theta(k^2,R).
\end{align}
Plugging this into the definition of $a$ shows that 
\begin{align}\begin{split}\label{eqnaasthephi}
     \E^{-\I k R} a(k) & =\frac{1-\I k R}{2} \theta(k^2,R) - \frac{1}{2\I k} \theta^\qd(k^2,R) \\
       & \qquad -\frac{\I k+ k^2 R}{2} \phi(k^2,R) +\frac{1}{2} \phi^\qd(k^2,R), \quad k\in\C_+. 
   \end{split}\end{align}
     In view of the second equality in~\eqref{eqncsatzero}, this identity guarantees that the function $a$ has a unique continuation to an entire function. 
\end{proof}
   
   Even more, the right-hand side of~\eqref{eqnaasthephi} is actually a polynomial in $k$ due to our assumptions on the supports of $\omega$ and $\dip$ on $[0,R]$. 
   This entails that the function $a$ has only a finite number of zeros, none of which lie on the real axis. 
   In fact, in order to verify this, we first define the function $b$ on $\C_+$ via 
  \begin{align}\label{eq:b_0}
    b(k) = \frac{\I k f(k,0) - f'(k,0-)}{2\I k}, \quad k\in\C_+.
  \end{align} 
  Similarly as before, we see that $b$ can be continued to an entire function because of the identity 
   \begin{align}\begin{split}\label{eqnbasthephi}
     \E^{-\I k R} b(k) & = - \frac{1 - \I k R}{2} \theta(k^2,R) + \frac{1}{2\I k} \theta^\qd(k^2,R) \\
       & \quad\qquad - \frac{\I k+ k^2 R}{2} \phi(k^2,R) + \frac{1}{2} \phi^\qd(k^2,R), \quad k\in\C_+. 
   \end{split}\end{align}
  Now it is a straightforward computation to verify that for real $k$, we have  
    \begin{align}\label{eq:absym}
    a(k)^\ast & = a(-k), &    b(k)^\ast & = b(-k),
     \end{align}
  as well as
  \begin{align}\label{eqnabone}
    |a(k)|^2  = |b(k)|^2 + 1,
  \end{align}
  which guarantees that $a$ has no zeros on the real axis. 
  Furthermore, all zeros in the upper complex half-plane $\C_+$ necessarily have to lie on the imaginary axis. 
  In fact, if $k$ was a zero in $\C_+$ that does not lie on the imaginary axis, then $k^2\in\C\backslash\R$ and 
  \begin{align}
    f'(k,0-) = -\I k f(k,0).
  \end{align}
  This would allow us to compute the imaginary part  
   \begin{align}
   \im\, m(k^2) = \im\, \frac{f'(k,0-)}{k^2f(k,0)} = \im\, \frac{1}{\I k} = -\frac{\re\, k}{|k|^2} = - \frac{\im\, k^2}{2|k|^2\im\,k},
   \end{align}
   which is a contradiction to the fact that $m$ is a Herglotz--Nevanlinna function.
  As this proves that all zeros in $\C_+$ indeed lie on the imaginary axis, we may enumerate them, repeated according to multiplicity (it can be shown that they are simple but we do not need this here), by $\I\kappa_1,\ldots,\I\kappa_N$ for some positive constants $\kappa_1,\ldots,\kappa_N$. 
  With this notation, let us state the following result. 
 
 \begin{lemma}\label{lemTF0}
   We have the identity 
  \begin{align}\label{eqnTF0}
      \frac{4}{3} \sum_{n=1}^N \frac{1}{\kappa_n^3} + \frac{2}{\pi} \int_{\R} \frac{1}{k^4} \log|a(k)| dk  =\int_0^\infty |\Wr(x) - x |^2 dx + \int_{[0,\infty)} d\dip. 
  \end{align}
 \end{lemma}
  
  \begin{proof}
   From the representation~\eqref{eqnaasthephi} for $a$, together with the formulas in Proposition~\ref{propFS} for the fundamental system $\theta$, $\phi$, we see that 
    \begin{align*}
      a(0) & = 1, & a'(0) & = 0, & a''(0) & = 0, 
    \end{align*}
    and after a cumbersome but straightforward computation furthermore that 
    \begin{align}\label{eqndddazero}
      a'''(0) = -3\I \biggl(\int_0^\infty |\Wr(x)-x|^2 dx +  \int_{[0,\infty)} d\dip\biggr).
    \end{align}
    In particular, this yields the Taylor expansion 
      \begin{align*}
        a(k) = 1 - k^3 \frac{\I}{2} \biggl(\int_0^\infty |\Wr(x)-x|^2 dx +  \int_{[0,\infty)} d\dip\biggr) + \OO(k^4), \qquad k\rightarrow0,
      \end{align*}
      around zero, which entails that       
        \begin{align}\label{eqnlogazero}
           \log|a(k)| = \OO(k^4)
       \end{align}
      as $k\rightarrow0$ on the real line. 
  
   Since the function $a$ is of bounded type in the upper complex half-plane, it admits a Nevanlinna factorization \cite[Theorem~6.13]{roro94} of the form 
   \begin{align*}
    a(k) = C  \prod_{n=1}^N \frac{\I \kappa_n - k}{\I \kappa_n + k}\exp\biggl\{-\I \beta k + \frac{1}{\pi\I} \int_\R \biggl(\frac{1}{t-k}-\frac{t}{1+t^2}\biggr)\log|a(t)|dt\biggr\}, \quad k\in\C_+,
   \end{align*}
  for some real constant $\beta\in\R$ (in fact, it is not difficult to show that $\beta = - R$) and a complex constant $C\in\C$ with modulus one.  
  Upon differentiating this, we obtain
  \begin{align*}
    \frac{a'(k)}{a(k)} = -\I\beta + \frac{1}{\pi\I} \int_\R \frac{1}{(t-k)^2} \log|a(t)|dt + \sum_{n=1}^N \frac{2\I \kappa_n}{\kappa_n^2+k^2}
  \end{align*}
  for all $k\in\C_+$ close enough to zero (so that $a(k)$ is non-zero). 
  After differentiating two more times, we get 
  \begin{align*}
    & \frac{a'''(k)}{a(k)} - 3\frac{a'(k)a''(k)}{a(k)^2}  + 2\frac{a'(k)^3}{a(k)^3} \\
      & \qquad\qquad =  \frac{6}{\pi\I} \int_\R \frac{1}{(t-k)^4} \log|a(t)|dt - {4}{\I}\sum_{n=1}^N \frac{\kappa_n(\kappa_n^2 - k^2)^2 - 4\kappa_n k^4}{(\kappa_n^2+k^2)^4},
  \end{align*}
  again, as long as $k\in\C_+$ is close enough to zero. 
   Now we obtain identity~\eqref{eqnTF0} upon letting $k$ tend to zero, employing~\eqref{eqndddazero} and noting that the limit of the integral on the right-hand side exists because of the asymptotics~\eqref{eqnlogazero}.
  \end{proof}

  Note that both terms on the left-hand side of the identity~\eqref{eqnTF0} are non-negative in view of~\eqref{eqnabone}. 
  In particular, this observation will allow us to obtain an estimate on the negative eigenvalues of the corresponding self-adjoint realization.
  To this end, we first point out that there are only finitely many such eigenvalues.
   More precisely, we see from~\eqref{eqnJosttp} and~\eqref{eqnJostdtp} that the right-hand side of~\eqref{eqnmJost} is a rational function. 
   This implies that the Weyl--Titchmarsh function $m$ has a continuation to a meromorphic function on $\C\backslash[0,\infty)$ with only finitely many poles. 
   As a consequence, the negative spectrum of $(L,\omega,\dip)$ consists only of finitely many eigenvalues. 
  Upon enumerating these eigenvalues by $\lambda_1,\ldots,\lambda_{K}$ with increasing modulus, we obtain the following Lieb--Thirring-type bound. 

  \begin{corollary}\label{corEVest}
   We have the estimate 
  \begin{align}
      \frac{4}{3} \sum_{i=1}^{K} \frac{1}{|\lambda_i|^{3/2}}  \leq \int_0^\infty |\Wr(x) - x |^2 dx + \int_{[0,\infty)} d\dip. 
  \end{align}
 \end{corollary}

 \begin{proof}
  We may assume that there are negative eigenvalues. 
  Since the function $m$ is a Herglotz--Nevanlinna function, we see from~\eqref{eqnmJost} that the function 
  \begin{align*}
    \kappa\mapsto -\frac{f'(\I\kappa,0-)}{\kappa^2 f(\I\kappa,0)}
  \end{align*}
  is real-valued, continuous and strictly decreasing for positive $\kappa$ away from the poles $\sqrt{-\lambda_1},\ldots,\sqrt{-\lambda_K}$. 
  Because of this, we can find a positive $\kappa<\sqrt{-\lambda_1}$ such that 
  \begin{align}\label{eqnfindkappa}
    -\frac{f'(\I\kappa,0-)}{\kappa^2 f(\I\kappa,0)} = -\frac{1}{\kappa}.
  \end{align}  
  Since this means that $\I\kappa$ is a zero of $a$, there is an index $n(1)\in\lbrace1,\ldots,N\rbrace$ such that $\kappa=\kappa_{n(1)}$ and thus $\kappa_{n(1)}<\sqrt{-\lambda_{1}}$. 
   If $K>1$ and $\lambda_{i-1}$, $\lambda_{i}$ are two consecutive eigenvalues for some $i\in\{2,\ldots,K\}$, then we can find a positive $\kappa$ between $\sqrt{-\lambda_{i-1}}$ and $\sqrt{-\lambda_{i}}$ such that~\eqref{eqnfindkappa} holds true.
  As before, we see that $\I\kappa$ is a zero of $a$ so that there is an index $n(i)\in\lbrace1,\ldots,N\rbrace$ such that $\kappa=\kappa_{n(i)}$ and thus also $\kappa_{n(i)}<\sqrt{-\lambda_{i}}$. 
  In conclusion, this shows that 
  \begin{align*}
    \frac{4}{3} \sum_{i=1}^{K} \frac{1}{|\lambda_i|^{3/2}} \leq \frac{4}{3} \sum_{i=1}^{K} \frac{1}{\kappa_{n(i)}^3} \leq \frac{4}{3} \sum_{n=1}^N \frac{1}{\kappa_n^3},
  \end{align*}
  which yields the claim upon invoking Lemma~\ref{lemTF0}. 
 \end{proof}

  The next ingredient for our proof will be an estimate on the absolutely continuous spectrum of $(L,\omega,\dip)$.
  To this end, we first note that we have    
\begin{align}\label{eq:m=ab}
   \I k\, m(k^2) = \frac{b(k)-a(k)}{b(k) + a(k)}
  \end{align}
  for all $k\in\C_+$ with $k^2\in\C\backslash\R$. 
Since the functions $a$ and  $b$ are entire and satisfy the properties  \eqref{eq:absym} and \eqref{eqnabone} on the real line, one can conclude that the spectrum of $(L,\omega,\dip)$ on the interval $[0,\infty)$ is purely absolutely continuous with the corresponding spectral measure $\mu$  given by 
  \begin{align}
    \mu(B) = \int_B \varrho(\lambda) d\lambda
  \end{align}
  for every Borel set $B\subseteq[0,\infty)$, where $\varrho$ is defined by 
   \begin{align}\label{eqnrho}
    \varrho(\lambda) = \lim_{\varepsilon\rightarrow0} \frac{1}{\pi}\im\, m(\lambda+\I\varepsilon) = \frac{1}{\pi\sqrt{\lambda}|b(\sqrt{\lambda}) + a(\sqrt{\lambda})|^2},
    \quad \lambda\in(0,\infty). 
   \end{align}
   We note that the function $\varrho$ is continuous and positive on $(0,\infty)$.
  
    \begin{corollary}\label{corACest}
   For every compact subset $\Omega$ of $(0,\infty)$, we have the estimate 
  \begin{align}\label{eqnTFest0}
      - \frac{1}{\pi}\int_\Omega \log\biggl(\varrho(\lambda) \frac{C_\Omega\lambda^3}{\sqrt{\lambda}}\biggl) \frac{\sqrt{\lambda}}{\lambda^3} d\lambda  \leq \int_0^\infty |\Wr(x) - x |^2 dx + \int_{[0,\infty)} d\dip,
  \end{align}
  where $C_\Omega = 4\pi (\min\,\Omega)^{-2}$ is a positive constant. 
 \end{corollary}
 
 \begin{proof}
   For every positive $k$, we first compute that 
     \begin{align*}
       \left| 1- \frac{b(k)-a(k)}{b(k)+a(k)}\right|^2 =  \frac{4 |a(k)|^2}{|b(k)+a(k)|^2} = 4\pi k |a(k)|^2  \varrho(k^2) 
  \end{align*}
  and on the other side that 
       \begin{align*}
       \left| 1- \frac{b(k)-a(k)}{b(k)+a(k)}\right|  \geq   \re\left( 1- \frac{b(k)-a(k)}{b(k)+a(k)}\right) = 1+ \frac{1}{|b(k)+a(k)|^2} \geq 1.
  \end{align*}
  In combination, this gives the bound 
  \begin{align*}
    \frac{1}{|a(\sqrt{\lambda})|^2}  \leq 4\pi \sqrt{\lambda} \varrho(\lambda) \leq C_\Omega \lambda^{5/2} \varrho(\lambda)
  \end{align*}  
  as long as $\lambda\in\Omega$, which allows us to estimate the integral  
  \begin{align*}
    - \frac{1}{\pi} \int_\Omega \log\biggl(\varrho(\lambda) \frac{C_\Omega\lambda^3}{\sqrt{\lambda}}\biggl) \frac{\sqrt{\lambda}}{\lambda^3} d\lambda \leq \frac{1}{\pi} \int_\Omega \log|a(\sqrt{\lambda})|^2 \frac{\sqrt{\lambda}}{\lambda^3}d\lambda.
   \end{align*}
   Upon employing a substitution, we can further bound this by 
   \begin{align*}
      \frac{2}{\pi} \int_{\sqrt{\min\Omega}}^{\sqrt{\max\Omega}} \log|a(k)|^2 \frac{1}{k^4} dk \leq \frac{4}{\pi} \int_{0}^{\infty} \log|a(k)| \frac{1}{k^4} dk = \frac{2}{\pi} \int_{\R} \log|a(k)| \frac{1}{k^4} dk,
  \end{align*}
  which yields the claim in view of Lemma~\ref{lemTF0}. 
 \end{proof}

  With these auxiliary facts, we are now in position to prove our first theorem. 

\begin{proof}[Proof of Theorem~\ref{thm0}]
  Let us assume for now that $S$ is a generalized indefinite string $(L,\omega,\dip)$ such that $L$ is infinite and 
    \begin{align*}
      \int_0^\infty \left| \Wr(x) - x\right|^2 dx + \int_{[0,\infty)}  d\dip & < \infty,
    \end{align*}
    where $\Wr$ is the normalized anti-derivative of $\omega$. 
  We are first going to construct a suitable approximating sequence of generalized indefinite strings $(L_n,\omega_n,\dip_n)$ from the set $\mathcal{F}_0$. 
  For every $n\in\N$, let $L_n$ be infinite and choose $R_n>n$ such that 
  \begin{align*}
    \int_{R_n}^\infty |\Wr(x)-x|^2 dx < \frac{1}{n}. 
  \end{align*}
  We can then find a real-valued function $\Wr_n$ on $[0,\infty)$ which is piecewise constant on the interval $[0,R_n]$ with 
  \begin{align*}
   \int_0^{R_n} |\Wr_n(x)-\Wr(x)|^2 dx < \frac{1}{n}
  \end{align*}
  and satisfies $\Wr_n(x)=x$ for all $x>R_n$. 
  The distribution $\omega_n$ is now defined in such a way that the corresponding normalized anti-derivative coincides with $\Wr_n$ almost everywhere. 
  Apart from this, we are able to find a non-negative Borel measure $\dip_n$ which is supported on a finite set contained in $[0,R_n)$ with 
  \begin{align*}
    \int_{[0,\infty)} d\dip_n = \int_{[0,\infty)} d\dip
  \end{align*}
  and such that 
  \begin{align*}
    \int_{[0,x)} d\dip_n \rightarrow \int_{[0,x)} d\dip, \qquad n\rightarrow \infty, 
  \end{align*}
  for almost every $x\in[0,\infty)$. 
  Note that by construction, we then have 
  \begin{align}\label{eqnNconv}
    \int_0^\infty |\Wr_n(x)-x|^2 dx + \int_{[0,\infty)} d\dip_n \rightarrow \int_0^\infty |\Wr(x)-x|^2 dx + \int_{[0,\infty)} d\dip
  \end{align}
  as $n\rightarrow\infty$.
  Furthermore, it follows readily from \cite[Proposition~6.2]{IndefiniteString} that the corresponding Weyl--Titchmarsh functions $m_n$ converge locally uniformly to $m$. 
  Thus the associated spectral measures $\mu_n$ certainly satisfy 
  \begin{align}\label{eqnmuconv}
   \int_\R g(\lambda)d\mu_n(\lambda) \rightarrow \int_\R g(\lambda) d\mu(\lambda), \qquad n\rightarrow\infty,
  \end{align}
  for every continuous function $g$ on $\R$ with compact support. 

   In order to prove that the essential spectrum of $S$ is restricted to $[0,\infty)$, let $I$ be a compact interval in $(-\infty,0)$. 
   Because of the estimate in Corollary~\ref{corEVest} and the convergence in~\eqref{eqnNconv}, we see that there is an integer $K_I$ such that $(L_n,\omega_n,\dip_n)$ has at most $K_I$ eigenvalues in the interval $I$ for every $n\in\N$. 
   It now follows from the convergence of the measures $\mu_n$ in~\eqref{eqnmuconv} that the limit measure $\mu$ is supported on a finite set on $I$, which implies that $S$ has at most finitely many eigenvalues in $I$.  
  Since the interval $I$ was arbitrary, we conclude that the essential spectrum of $S$ is necessarily contained in $[0,\infty)$. 
  
  Now take a compact set $\Omega\subset(0,\infty)$ of positive Lebesgue measure.
   Due to the convergence of the measures $\mu_n$ in~\eqref{eqnmuconv} we have (see \cite[Theorem~30.2]{ba01})
  \begin{align*}
   \mu(\Omega) \geq \limsup_{n\rightarrow\infty} \mu_n(\Omega) = \limsup_{n\rightarrow\infty} \int_\Omega \varrho_n(\lambda)d\lambda,
  \end{align*}
  where the functions $\varrho_n$ are given as in~\eqref{eqnrho}. 
 An application of Jensen's inequality \cite[Theorem~3.3]{ru74} then furthermore yields
 \begin{align*}
  \mu(\Omega) & \geq \limsup_{n\rightarrow\infty} D_\Omega \exp\biggl\{\frac{1}{C_\Omega D_\Omega} \int_\Omega \log\biggl(\varrho_n(\lambda)\frac{C_\Omega\lambda^3}{\sqrt{\lambda}}\biggr) \frac{\sqrt{\lambda}}{\lambda^3}  d\lambda\biggr\}, 
 \end{align*}
 where $C_\Omega$, $D_\Omega$ are positive constants defined as in Corollary~\ref{corACest} and by
 \begin{align*}
  D_\Omega= \frac{1}{C_\Omega} \int_\Omega   \frac{\sqrt{\lambda}}{\lambda^3}  d\lambda.
 \end{align*}
  In view of the estimate in Corollary~\ref{corACest} and the convergence in~\eqref{eqnNconv}, we can conclude that $\mu(\Omega)$ is indeed positive with
   \begin{align*}
  \mu(\Omega) & \geq   D_\Omega \exp\biggl\{\frac{-\pi}{C_\Omega D_\Omega} \biggl(\int_0^\infty |\Wr(x)-x|^2dx + \int_{[0,\infty)}d\dip\biggr)\biggr\}.
 \end{align*}
  Since all Borel measures on $\R$ are regular, this readily implies that $\mu(\Omega)$ is positive for every Borel set $\Omega\subseteq[0,\infty)$ of positive Lebesgue measure. 
  With this fact, we have finally verified that the essential spectrum of $S$ coincides with the interval $[0,\infty)$ and the absolutely continuous spectrum of $S$ is essentially supported on $[0,\infty)$. 
  
  In order to finish the proof of Theorem~\ref{thm0}, let us suppose that $S$ is a generalized indefinite string $(L,\omega,\dip)$ such that  $L$ is infinite and~\eqref{eqnCondS0} holds   for a real constant $c$ and a positive constant $\eta$.
  We consider the generalized indefinite string $({L},\tilde{\omega},\tilde{\dip})$, where $\tilde{\omega}$ is defined via its normalized anti-derivative $\tilde{\Wr}$ by 
  \begin{align}\label{eqnWrtilde} 
    \tilde{\Wr} & = \frac{\Wr - c}{\eta}
  \end{align}
  and $\tilde{\dip} = \eta^{-2} \dip$.
  Since $(L,\tilde{\omega},\tilde{\dip})$ satisfies the assumptions imposed before, we infer that the essential spectrum of $(L,\tilde{\omega},\tilde{\dip})$ coincides with the interval $[0,\infty)$ and its absolutely continuous spectrum is essentially supported on $[0,\infty)$. 
  However, since the corresponding Weyl--Titchmarsh functions $m$ and $\tilde{m}$ are related via 
  \begin{align}\label{eqnmtilde}
    m(z) = \eta\, \tilde{m}(\eta z) + c, \quad z\in\C\backslash\R, 
  \end{align} 
  the same is true for $S$.
\end{proof}

 \section{Proof of Theorem~\ref{thmalpha}}\label{secPr2}

  In order to prepare for the proof of our second main result, let us fix a positive constant $\alpha$.
  We assume that $(L,\omega,\dip)$ is a generalized indefinite string such that $L$ is infinite, there is an $R>0$ such that the normalized anti-derivative $\Wr$ of $\omega$ satisfies 
  \begin{align}
    \Wr(x) = \frac{x}{1+2\sqrt{\alpha}x} 
  \end{align}
  for almost all $x$ in $[R,\infty)$ and the measure $\dip$ vanishes on $[R,\infty)$. 
   In addition, let us also suppose that $\Wr$ is equal to a piecewise constant function almost everywhere on the interval $[0,R]$ and that the support of the measure $\dip$ is a finite set. 
   For easy reference later on, we will denote the set of generalized indefinite strings defined in this way by $\F_\alpha$. 
   Under our assumptions, for every $k\in\C_+$, there is a {\em Jost solution}  $f(k,\redot)$ of the differential equation~\eqref{eqnDEho}  with $z=k^2+\alpha\in\C\backslash[\alpha,\infty)$ such that 
  \begin{align}
    f(k,x) = (1+2\sqrt{\alpha}x)^{\frac{\I k}{2\sqrt{\alpha}} + \frac{1}{2}}, \quad x\in[R,\infty).
  \end{align} 
  We note that since the function $f(k,\redot)$ clearly lies in $\dot{H}^1[0,\infty)$ and $L^2([0,\infty);\dip)$, the corresponding Weyl--Titchmarsh function $m$ is given by 
   \begin{align}\label{eqnmJostalpha}
     m(k^2+\alpha) =  \frac{f'(k,0-)}{(k^2+\alpha)f(k,0)}
    \end{align} 
    as long as $k^2+\alpha\in\C\backslash\R$.  
    Furthermore, we define the function $a$ on $\C_+$ via 
   \begin{align}
     a(k) = \frac{(\I k - \sqrt{\alpha})f(k,0)+f'(k,0-)}{2\I k}, \quad k\in\C_+. 
   \end{align}
   If we denote with $\theta$, $\phi$ the fundamental system of solutions of the differential equation~\eqref{eqnDEho} as in Section~\ref{secES}, then we readily  compute that 
   \begin{align}
   \begin{split}\label{eqnfptalpha}
     f(k,0) & = (1+2\sqrt{\alpha}R)^{\frac{\I k}{2\sqrt{\alpha}}+\frac{1}{2}} \phi^\qd(k^2+\alpha,R) \\
       & \qquad - (\I k+\sqrt{\alpha})(1-(\I k-\sqrt{\alpha})R)(1+2\sqrt{\alpha}R)^{\frac{\I k}{2\sqrt{\alpha}}-\frac{1}{2}} \phi(k^2+\alpha,R), \end{split} \\
    \begin{split}\label{eqnfqdptalpha}
     f'(k,0-) & = - (1+2\sqrt{\alpha}R)^{\frac{\I k}{2\sqrt{\alpha}}+\frac{1}{2}} \theta^\qd(k^2+\alpha,R) \\
       & \qquad + (\I k+\sqrt{\alpha})(1-(\I k-\sqrt{\alpha})R)(1+2\sqrt{\alpha}R)^{\frac{\I k}{2\sqrt{\alpha}}-\frac{1}{2}} \theta(k^2+\alpha,R), \end{split}
   \end{align}
   which yields, upon plugging these expressions into the definition of $a$, that 
    \begin{align}\begin{split}\label{eqnaasthephi_a}
    &  \frac{2\I k}{(1+2\sqrt{\alpha}R)^{\frac{\I k}{2\sqrt{\alpha}} + \frac{1}{2}}}a(k) \\
     & \quad  =\frac{(\I k + \sqrt{\alpha})(1 - (\I k - \sqrt{\alpha})R)}{1+2\sqrt{\alpha}R}\theta(k^2+\alpha,R) - \theta^\qd(k^2+\alpha,R) \\
       & \qquad\quad +  \frac{1-(\I k - \sqrt{\alpha})R}{1+2\sqrt{\alpha}R}(k^2+\alpha)\phi(k^2+\alpha,R)   + (\I k - \sqrt{\alpha})\phi^\qd(k^2+\alpha,R) 
   \end{split}\end{align}
   for all $k\in\C_+$.
   Due to our assumptions on the supports of $\omega$ and $\dip$ on $[0,R]$, the right-hand side of this equation turns out to be a polynomial in $k$.
   This guarantees that the function $a$ admits an analytic continuation (denoted with $a$ as well for simplicity) to all of $\C$ except for zero, where $a$ has at most a simple pole. 
   It furthermore entails that $a$ has only finitely many zeros, none of which lie on the real axis. 
   In order to prove this, we introduce the function $b$ on $\C_+$ next via 
   \begin{align}
     b(k) = \frac{(\I k +  \sqrt{\alpha})f(k,0) - f'(k,0-)}{2\I k}, \quad k\in\C_+.
   \end{align}
   One infers from the expressions in~\eqref{eqnfptalpha} and~\eqref{eqnfqdptalpha} that $b$ also admits an analytic continuation to all of $\C$ except for zero.
   Now after a straightforward computation, we see that for all non-zero real $k$, we have  
   \begin{align}\label{eqnabalpha}
     a(k)^\ast & = a(-k), & b(k)^\ast & = b(-k), 
   \end{align}  
   as well as 
   \begin{align}\label{eqnapbalpha}
     |a(k)|^2 = |b(k)|^2 + 1, 
   \end{align}
   which guarantees that $a$ has no zeros on the real axis.
   Moreover, all zeros in $\C_+$ necessarily have to lie on the imaginary axis.  
   In fact, if $k$ was a zero in $\C_+$ that does not lie on the imaginary axis, then $k^2+\alpha\in\C\backslash\R$ and 
   \begin{align}
     f'(k,0-) = -(\I k - \sqrt{\alpha})f(k,0).
   \end{align}
   This would allow us to compute the imaginary part
   \begin{align}\begin{split}
     \im\, m(k^2+\alpha) & = \im\, \frac{f'(k,0-)}{(k^2+\alpha) f(k,0)} = \im\, \frac{1}{\I k+\sqrt{\alpha}} = - \frac{\re\, k}{|\I k+\sqrt{\alpha}|^2} \\
      & = - \frac{\im\, (k^2+\alpha)}{2\im\, k\,|\I k+\sqrt{\alpha}|^2},
   \end{split}\end{align}
   which is a contradiction to the fact that $m$ is a Herglotz--Nevanlinna function. 
   As this proves that all zeros in $\C_+$ indeed lie on the imaginary axis, we may enumerate them, repeated according to multiplicity, by $\I\kappa_1,\ldots,\I\kappa_N$ for some positive constants $\kappa_1,\ldots,\kappa_N$. 
   The next result collects two trace formulas which are going to play a key role in the proof of Theorem \ref{thmalpha}.

  \begin{lemma}\label{lemTF2}
   We have the identities 
  \begin{align}\label{eq:trace01alpha}
  \begin{split}
   &   \frac{1}{\sqrt{\alpha}}\sum_{n=1}^N \frac{ \kappa_n}{\kappa_n^2-\alpha} + \frac{1}{2{\alpha}}\sum_{n=1}^N\log \left|\frac{\kappa_n - \sqrt{\alpha}}{\kappa_n + \sqrt{\alpha}}\right|  +  \frac{\sqrt{\alpha}}{\pi} \int_\R \frac{1}{(k^2+{\alpha})^2} \log|a(k)|dk   \\
    & \qquad=  \int_0^\infty \Wr(x) - \frac{x}{1+2\sqrt{\alpha}x}\,dx  
    \end{split}\end{align}
   and
   \begin{align}\label{eq:trace02alpha}
  \begin{split}
     &   \frac{1}{2{\alpha}}\sum_{n=1}^N \frac{\kappa_n^3+\alpha\kappa_n }{(\kappa_n^2-\alpha)^2} + \frac{1}{4{\alpha}^{3/2}}\sum_{n=1}^N\log \left|\frac{\kappa_n - \sqrt{\alpha}}{\kappa_n + \sqrt{\alpha}}\right|+ \frac{2}{\pi} \int_\R \frac{k^2}{(k^2+{\alpha})^3} \log|a(k)|dk \\
     &  \qquad =\int_0^\infty  \Bigl|\Wr(x) - \frac{x}{1+2\sqrt{\alpha}x} \Bigr|^2 (1+2\sqrt{\alpha}x) dx + \int_{[0,\infty)} (1+2\sqrt{\alpha}x) d\dip(x).  
   \end{split}\end{align}
 \end{lemma}
  
  \begin{proof}
  It follows readily from~\eqref{eqnaasthephi_a} and~\eqref{eqncsatzero} that 
   \begin{align*}
     a(\I\sqrt{\alpha}) = 1. 
   \end{align*}
   By differentiating both sides in~\eqref{eqnaasthephi_a}, evaluating at $\I\sqrt{\alpha}$ and using the expressions from Proposition~\ref{propFS}, we also see that  
      \begin{align*}
         a'(\I\sqrt{\alpha}) & = 2\I \sqrt{\alpha}\int_0^\infty \Wr(x) - \Wr_{\alpha}(x)\,dx,
      \end{align*}
     where the function $\Wr_\alpha$ is simply given by  
    \begin{align*}
    \Wr_\alpha(x) = \frac{x}{1+2\sqrt{\alpha}x}, \quad x\in[0,\infty),
  \end{align*}
  as in Example~\ref{exaalpha}. 
      After similar but longer computations, we furthermore get 
      \begin{align*}
         \begin{split} a''(\I\sqrt{\alpha}) & = 4{\sqrt{\alpha}} \int_0^\infty (\Wr(x)-\Wr_{\alpha}(x))^2 (1+2\sqrt{\alpha}x) dx + 4{\sqrt{\alpha}} \int_{[0,\infty)} (1+2\sqrt{\alpha}x) d\dip(x) \\ 
         & \qquad - 4\alpha \biggl(\int_0^\infty \Wr(x) - \Wr_{\alpha}(x)\, dx \biggr)^2  - 2 \int_0^\infty \Wr(x) - \Wr_{\alpha}(x)\, dx. \end{split}
      \end{align*}

 Since the function $a$ is of bounded type in the upper complex half-plane, it admits a Nevanlinna factorization \cite[Theorem~6.13]{roro94} of the form 
   \begin{align}\label{eq:intrepaalpha}
     a(k) = C  \prod_{n=1}^N \frac{\I \kappa_n - k}{\I \kappa_n + k} \exp\biggl\{-\I \beta k + \frac{\I\gamma}{k} + \frac{1}{\pi\I} \int_\R \biggl(\frac{1}{t-k}-\frac{t}{1+t^2}\biggr)\log|a(t)|dt \biggr\} 
   \end{align}
   for all $k\in\C_+$, where $\beta$ and $\gamma$ are real constants and $C$ is a complex constant with modulus one.   
   However, as zero is at most a simple pole of $a$, we infer that the constant $\gamma$ has to be equal to zero.  
      After differentiating \eqref{eq:intrepaalpha}, we obtain
  \begin{align}\label{eq:1st}
    \frac{a'(k)}{a(k)} = \sum_{n=1}^N \frac{2\I \kappa_n}{\kappa_n^2+k^2} -\I\beta + \frac{1}{\pi\I} \int_\R \frac{1}{(t-k)^2} \log|a(t)|dt 
  \end{align}
  for all $k\in\C_+$ close enough to $\I\sqrt{\alpha}$ (so that $a(k)$ is non-zero). 
 Upon differentiating one more time, we also get 
  \begin{align}\label{eq:2nd}
     \frac{a''(k)}{a(k)} - \frac{a'(k)^2}{a(k)^2} = \sum_{n=1}^N \frac{ - 4\I \kappa_n k}{(\kappa_n^2+k^2)^2} +  \frac{2}{\pi\I} \int_\R \frac{1}{(t-k)^3} \log|a(t)|dt,
  \end{align}
  again, as long as $k\in\C_+$ is close enough to $\I\sqrt{\alpha}$. 
   
 Now we evaluate both sides of \eqref{eq:intrepaalpha} at $\I\sqrt{\alpha}$ and take absolute values to obtain
\begin{align*}
1 =  \prod_{n=1}^N \left|\frac{\kappa_n - \sqrt{\alpha}}{\kappa_n + \sqrt{\alpha}}\right| \exp\biggl\{ \beta \sqrt{\alpha} + \frac{\sqrt{\alpha}}{\pi} \int_\R \frac{1}{t^2+{\alpha}}\log|a(t)|dt\biggr\}.
\end{align*}
This gives the following expression for $\beta$ in terms of the scattering data:
\begin{align*}
\beta =  \frac{1}{\sqrt{\alpha}}\sum_{n=1}^N\log \left|\frac{\kappa_n + \sqrt{\alpha}}{\kappa_n - \sqrt{\alpha}}\right| - \frac{1}{\pi} \int_\R \frac{1}{t^2+{\alpha}}\log|a(t)|dt.
\end{align*}
 By evaluating equality~\eqref{eq:1st} at $\I\sqrt{\alpha}$, we next get 
  \begin{align*}
     2\I \sqrt{\alpha}\int_0^\infty \Wr(x) - \Wr_{\alpha}(x)\,dx = & \sum_{n=1}^N \frac{2\I \kappa_n}{\kappa_n^2-\alpha} -\I\beta + \frac{1}{\pi\I} \int_\R \frac{1}{(t-\I\sqrt{\alpha})^2} \log|a(t)|dt. 
      \end{align*}
  Taking imaginary parts and using the expression for $\beta$ yields \eqref{eq:trace01alpha}.
 In a similar manner, the real part of the right-hand side of~\eqref{eq:2nd} evaluated at $\I\sqrt{\alpha}$ becomes 
 \begin{align*}
      \sum_{n=1}^N \frac{4\sqrt{\alpha} \kappa_n }{(\kappa_n^2-\alpha)^2} +  \frac{2}{\pi} \int_\R \frac{3t^2\sqrt{\alpha} - \alpha^{3/2}}{(t^2+{\alpha})^3} \log|a(t)|dt
      \end{align*}
      and the left-hand side 
            \begin{align*}
           & 4{\sqrt{\alpha}} \int_0^\infty (\Wr(x)-\Wr_{\alpha}(x))^2 (1+2\sqrt{\alpha}x) dx + 4{\sqrt{\alpha}} \int_{[0,\infty)} (1+2\sqrt{\alpha}x) d\dip(x) \\ 
         & \qquad  - \frac{2}{\sqrt{\alpha}}\sum_{n=1}^N \frac{ \kappa_n}{\kappa_n^2-\alpha} - \frac{1}{\alpha}\sum_{n=1}^N\log \left|\frac{\kappa_n - \sqrt{\alpha}}{\kappa_n + \sqrt{\alpha}}\right|  -  \frac{2\sqrt{\alpha}}{\pi} \int_\R \frac{1}{(t^2+{\alpha})^2} \log|a(t)|dt, 
      \end{align*}
      where we also made use of equality~\eqref{eq:trace01alpha}. 
      After equating these two expressions, we readily end up with~\eqref{eq:trace02alpha}.
  \end{proof}

We note that the integral term on the left-hand side of the identity~\eqref{eq:trace02alpha} is non-negative in view of~\eqref{eqnapbalpha}.
In particular, this observation will allow us to obtain a Lieb--Thirring-type estimate on the eigenvalues of the corresponding self-adjoint realization below $\alpha$.
To this end, we first point out that there are only finitely many such eigenvalues.
More precisely, we see from~\eqref{eqnfptalpha} and~\eqref{eqnfqdptalpha} that the right-hand side of~\eqref{eqnmJostalpha} is a rational function. 
This implies that the Weyl--Titchmarsh function $m$ has a continuation to a meromorphic function on $\C\backslash[\alpha,\infty)$ with only finitely many poles, none of which is located at zero. 
 Hence, we may enumerate all eigenvalues below $\alpha$ in the following way:
  \begin{align}
    \lambda_{K_-}^- < \dots < \lambda_{1}^- < 0 < \lambda_{1}^+ <\dots < \lambda_{K_+}^+ <\alpha.
\end{align}

\begin{corollary}\label{cor:LTalpha}
We have the estimate
\begin{align}\label{eq:LTalpha}
  \begin{split}
     & \frac{4}{3 \alpha^{3/2}}\sum_{i=1}^{K_-} \biggl(1-\frac{\lambda_i^-}{{\alpha}}\biggr)^{-3/2} + \frac{4}{3\alpha^{3/2}}\sum_{i=1}^{K_+} \biggl(1-\frac{\lambda_i^+}{{\alpha}}\biggr)^{3/2} \\
     & \qquad  \leq \int_0^\infty  \Bigl|\Wr(x) - \frac{x}{1+2\sqrt{\alpha}x} \Bigr|^2 (1+2\sqrt{\alpha}x) dx + \int_{[0,\infty)} (1+2\sqrt{\alpha}x) d\dip(x).  
   \end{split}\end{align}
\end{corollary}
 
 \begin{proof}
  Since the function $m$ is a Herglotz--Nevanlinna function, we see from~\eqref{eqnmJostalpha} that the function
  \begin{align*}
    \kappa \mapsto \frac{f'(\I\kappa,0-)}{(-\kappa^2+\alpha) f(\I\kappa,0)}
  \end{align*}
  is real-valued, continuous and strictly decreasing for positive $\kappa$ away from the poles 
  \begin{align*}
    \sqrt{\alpha-\lambda_{K_+}^+},\ldots,\sqrt{\alpha-\lambda_{1}^+},\sqrt{\alpha-\lambda_1^-},\ldots,\sqrt{\alpha-\lambda_{K_-}^-}.
  \end{align*}
  Because of this, we can find a positive $\kappa$ such that 
  \begin{align*}
    \frac{f'(\I\kappa,0-)}{(-\kappa^2+\alpha) f(\I\kappa,0)} = \frac{1}{\sqrt{\alpha}-\kappa}
  \end{align*} 
 between each pair of consecutive points in the sequence
   \begin{align*}
    \sqrt{\alpha-\lambda_{K_+}^+},\ldots,\sqrt{\alpha-\lambda_{1}^+},\sqrt{\alpha},\sqrt{\alpha-\lambda_1^-},\ldots,\sqrt{\alpha-\lambda_{K_-}^-}.
  \end{align*} 
 As $\I\kappa$ is a zero of the function $a$ for each such $\kappa$, we conclude that 
 \begin{align*}
    \sqrt{\alpha-\lambda_{K_+}^+}<\kappa_{n_+(K_+)}<\sqrt{\alpha-\lambda_{K_+-1}^+} < \cdots < \sqrt{\alpha-\lambda_{1}^+} < \kappa_{n_+(1)} < \sqrt{\alpha} 
    \end{align*}
    as well as
    \begin{align*}
    \sqrt{\alpha} < \kappa_{n_-(1)} < \sqrt{\alpha-\lambda_1^-} <\cdots<\sqrt{\alpha-\lambda_{K_--1}^-} < \kappa_{n_-(K_-)} < \sqrt{\alpha-\lambda_{K_-}^-}
  \end{align*}
  for some indices $n_+(K_+),\ldots,n_+(1),n_-(1),\ldots,n_-(K_-)\in\{1,\ldots,N\}$.
 
  Now let us consider the function $F$ defined by 
  \begin{align}\label{eq:funtrace}
    F(s) =  {2} \frac{s^3+s}{(s^2-1)^2} +  \log \left|\frac{s - 1}{s + 1}\right|, \quad s\in(0,1)\cup(1,\infty),
 \end{align}
 and first notice that 
\begin{align*}
F(1/s) = F(s), \quad s\in(0,1)\cup(1,\infty).
\end{align*}
 It is also straightforward to see that $F$ is strictly increasing on $(0,1)$ and strictly decreasing on $(1,\infty)$ since we may compute 
\begin{align*}
 F'(s) = -\frac{16s^2}{(s^2-1)^3}, \quad s\in(0,1)\cup(1,\infty). 
\end{align*}
 Moreover, the function $F$ satisfies the bound
\begin{align*}
F(s) = \int_s^\infty \frac{16 r^2}{(r^2-1)^3} dr \geq \int_s^\infty \frac{16}{r^4} dr = \frac{16}{3 s^3} > 0, \quad s\in(1,\infty). 
\end{align*}
 By combining all these facts, we can estimate 
    \begin{align*}
  &  \biggl(1-\frac{\lambda_i^\pm}{{\alpha}}\biggr)^{\pm3/2} <   \biggl(\frac{\kappa_{n_\pm(i)}}{\sqrt{\alpha}}\biggr)^{\pm 3} \leq  \frac{3}{16} F\biggl(\frac{\kappa_{n_\pm(i)}}{\sqrt{\alpha}}\biggr)
 \end{align*}
 for all $i\in\{1,\ldots,K_\pm\}$. 
 This allows us to bound the left-hand side of~\eqref{eq:LTalpha} by 
 \begin{align*}
    \frac{1}{4\alpha^{3/2}} \sum_{i=1}^{K_-} F\biggl(\frac{\kappa_{n_-(i)}}{\sqrt{\alpha}}\biggr) + \frac{1}{4\alpha^{3/2}} \sum_{i=1}^{K_+} F\biggl(\frac{\kappa_{n_+(i)}}{\sqrt{\alpha}}\biggr)  \leq  \frac{1}{4\alpha^{3/2}} \sum_{n=1}^{N} F\biggl(\frac{\kappa_{n}}{\sqrt{\alpha}}\biggr).
 \end{align*}
 It is readily seen that this sum coincides with the first two terms in~\eqref{eq:trace02alpha}, which yields the claim as the integral term on the left-hand side there is non-negative.
\end{proof}

We are now going to use the identity~\eqref{eq:trace02alpha} to estimate the absolutely continuous spectrum of  $(L,\omega,\dip)$. 
 To this end, we first note that we have  
\begin{align}\label{eq:m=ab2}
 (k^2+\alpha)m(k^2+\alpha) = \sqrt{\alpha}-\I k\,\frac{b(k)-a(k)}{b(k) + a(k)} 
\end{align}
 for all $k\in\C_+$ with $k^2+\alpha\in\C\backslash\R$. 
 Since the functions $a$ and $b$ are analytic on all of $\C$ except for zero and satisfy the properties~\eqref{eqnabalpha} and~\eqref{eqnapbalpha} on the real line, one can conclude that the spectrum of $(L,\omega,\dip)$ on the interval $(\alpha,\infty)$ is purely absolutely continuous with the corresponding spectral measure $\mu$ given by 
  \begin{align}
    \mu(B) = \int_B \varrho(\lambda) d\lambda,
  \end{align}
  for every Borel set $B\subseteq(\alpha,\infty)$, where $\varrho$ is defined by 
   \begin{align}\label{eqnrhoabalpha}
    \varrho(\lambda) = \lim_{\varepsilon\rightarrow0} \frac{1}{\pi}\im\, m(\lambda+\I\varepsilon) 
    = \frac{\sqrt{\lambda-\alpha}}{\pi \lambda|b(\sqrt{\lambda - \alpha}) + a(\sqrt{\lambda - \alpha})|^2},   \quad \lambda\in(\alpha,\infty). 
   \end{align}
   We note that the function $\varrho$ is continuous and positive on $(\alpha,\infty)$.
   
\begin{corollary}\label{corACest2}
For every compact subset $\Omega$ of $(\alpha,\infty)$, we have the estimate 
  \begin{align}\label{eqnTFest02}
  \begin{split}
     &  - \frac{1}{\pi}\int_\Omega \log\biggl(\varrho(\lambda) \frac{4\pi\lambda^3}{\alpha^2\sqrt{\lambda-\alpha}}\biggl) \frac{\sqrt{\lambda-\alpha}}{\lambda^3} d\lambda  \\
    & \qquad  \leq  \int_0^\infty \Bigl|\Wr(x) - \frac{x}{1+2\sqrt{\alpha}x} \Bigr|^2  (1+2\sqrt{\alpha}x) dx  + \int_{[0,\infty)} (1+2\sqrt{\alpha}x) d\dip(x).
       \end{split}
  \end{align}
 \end{corollary} 
 
 \begin{proof}
   For every positive $k$, we first compute that 
     \begin{align*}
       \left| 1- \frac{b(k)-a(k)}{b(k)+a(k)}\right|^2 =  \frac{4 |a(k)|^2}{|b(k)+a(k)|^2} = 4\pi \frac{k^2+\alpha}{k} \varrho(k^2+\alpha) |a(k)|^2
  \end{align*}
  and on the other side that 
       \begin{align*}
       \left| 1- \frac{b(k)-a(k)}{b(k)+a(k)}\right|  \geq   \re\left( 1- \frac{b(k)-a(k)}{b(k)+a(k)}\right) = 1+\frac{1}{|b(k)+a(k)|^2} \geq 1.
  \end{align*}
  In combination, this gives the bound 
  \begin{align*}
    \frac{1}{|a(\sqrt{\lambda-\alpha})|^2}  \leq  \frac{4\pi\lambda}{\sqrt{\lambda-\alpha}} \varrho(\lambda) \leq   \frac{4\pi\lambda^3}{\alpha^2\sqrt{\lambda-\alpha}} \varrho(\lambda)
  \end{align*}  
  as long as $\lambda\in\Omega$, which allows us to estimate the integral  
  \begin{align*}
    - \frac{1}{\pi} \int_\Omega \log\biggl(\varrho(\lambda) \frac{4\pi\lambda^3}{\alpha^2\sqrt{\lambda-\alpha}}\biggl) \frac{\sqrt{\lambda-\alpha}}{\lambda^3} d\lambda \leq \frac{1}{\pi} \int_\Omega \log|a(\sqrt{\lambda-\alpha})|^2 \frac{\sqrt{\lambda-\alpha}}{\lambda^3}d\lambda.
   \end{align*}
   Upon employing a substitution, we can further bound this by 
   \begin{align*}
       \frac{2}{\pi} \int_{\sqrt{\min\Omega-\alpha}}^{\sqrt{\max\Omega-\alpha}} \log|a(k)|^2 \frac{k^2}{(k^2+\alpha)^3} dk \leq \frac{2}{\pi} \int_{\R} \log|a(k)| \frac{k^2}{(k^2+\alpha)^3} dk.
  \end{align*}
 Now it remains to notice that the sum of the first two terms in~\eqref{eq:trace02alpha} is non-negative since the function $F$ defined in~\eqref{eq:funtrace} takes positive values. 
 \end{proof}

 We are now ready to prove our second main result.
     
\begin{proof}[Proof of Theorem~\ref{thmalpha}]
   Let us assume for now that $S$ is a generalized indefinite string $(L,\omega,\dip)$ such that $L$ is infinite and 
    \begin{align*}
      \int_0^\infty \Bigl| \Wr(x) - \frac{x}{1+2\sqrt{\alpha}x}\Bigr|^2 x\, dx + \int_{[0,\infty)}  x\, d\dip(x) & < \infty,
    \end{align*}
    where $\Wr$ is the normalized anti-derivative of $\omega$. 
  We are first going to construct a suitable approximating sequence of generalized indefinite strings $(L_n,\omega_n,\dip_n)$ from the set $\mathcal{F}_\alpha$. 
  For every $n\in\N$, let $L_n$ be infinite and choose $R_n>n$ such that 
  \begin{align*}
    \int_{R_n}^\infty \Bigl|\Wr(x)-\frac{x}{1+2\sqrt{\alpha}x}\Bigr|^2 x\, dx < \frac{1}{n}. 
  \end{align*}
  We can then find a real-valued function $\Wr_n$ on $[0,\infty)$ which is piecewise constant on the interval $[0,R_n]$ with 
  \begin{align*}
   \int_0^{R_n} |\Wr_n(x)-\Wr(x)|^2 dx & < \frac{1}{n R_n}
  \end{align*}
  and which is given explicitly by 
  \begin{align*}
    \Wr_n(x) & = \frac{x}{1+2\sqrt{\alpha}x}
   \end{align*}
   for all $x>R_n$. 
  The distribution $\omega_n$ is now defined in such a way that the corresponding normalized anti-derivative coincides with $\Wr_n$ almost everywhere. 
  Apart from this, we are able to find a non-negative Borel measure $\dip_n$ which is supported on a finite set contained in $[0,R_n)$ with 
  \begin{align*}
   \int_{[0,\infty)} d\dip_n & = \int_{[0,\infty)} d\dip, & \int_{[0,\infty)} x\, d\dip_n(x) & \leq \int_{[0,\infty)} x\, d\dip(x)
  \end{align*}
  and such that for almost every $x\in[0,\infty)$ we have 
  \begin{align*}
    \int_{[0,x)} d\dip_n \rightarrow \int_{[0,x)} d\dip, \qquad n\rightarrow \infty.
  \end{align*} 
  Note that by construction, there is a positive constant $M$ such that  
  \begin{align}\label{eqnNconvalpha}
     \int_0^\infty \Bigl|\Wr_n(x)-\frac{x}{1+2\sqrt{\alpha}x}\Bigr|^2 (1+2\sqrt{\alpha}x) dx + \int_{[0,\infty)} (1+2\sqrt{\alpha}x) d\dip_n(x) \leq M 
    \end{align}
    for all $n\in\N$. 
  Furthermore, it follows readily from \cite[Proposition~6.2]{IndefiniteString} that the corresponding Weyl--Titchmarsh functions $m_n$ converge locally uniformly to $m$. 
  Thus the associated spectral measures $\mu_n$ certainly satisfy 
  \begin{align}\label{eqnmuconvalpha}
   \int_\R g(\lambda)d\mu_n(\lambda) \rightarrow \int_\R g(\lambda) d\mu(\lambda), \qquad n\rightarrow\infty,
  \end{align}
  for every continuous function $g$ on $\R$ with compact support. 

   In order to prove that the essential spectrum of $S$ is restricted to $[\alpha,\infty)$, let $I$ be a compact interval in $(-\infty,\alpha)$. 
   Because of the estimate in Corollary~\ref{cor:LTalpha} and the bound in~\eqref{eqnNconvalpha}, we see that there is an integer $K_I$ such that $(L_n,\omega_n,\dip_n)$ has at most $K_I$ eigenvalues in the interval $I$ for every $n\in\N$. 
   It now follows from the convergence of the measures $\mu_n$ in~\eqref{eqnmuconvalpha} that the limit measure $\mu$ is supported on a finite set on $I$, which implies that $S$ has at most finitely many eigenvalues in $I$.  
  Since the interval $I$ was arbitrary, we conclude that the essential spectrum of $S$ is necessarily contained in $[\alpha,\infty)$. 
  
  Now take a compact set $\Omega\subset(\alpha,\infty)$ of positive Lebesgue measure.
   Due to  the convergence of the measures $\mu_n$ in~\eqref{eqnmuconvalpha} we have (see \cite[Theorem~30.2]{ba01})
  \begin{align*}
   \mu(\Omega) \geq \limsup_{n\rightarrow\infty} \mu_n(\Omega) = \limsup_{n\rightarrow\infty} \int_\Omega \varrho_n(\lambda)d\lambda,
  \end{align*}
  where the functions $\varrho_n$ are given as in~\eqref{eqnrhoabalpha}. 
 An application of Jensen's inequality \cite[Theorem~3.3]{ru74} then furthermore yields
 \begin{align*}
  \mu(\Omega) & \geq \limsup_{n\rightarrow\infty} D_\Omega \exp\biggl\{\frac{\alpha^2}{4\pi D_\Omega} \int_\Omega \log\biggl(\varrho_n(\lambda)\frac{4\pi\lambda^3}{\alpha^2\sqrt{\lambda-\alpha}}\biggr) \frac{\sqrt{\lambda-\alpha}}{\lambda^3}  d\lambda\biggr\}, 
 \end{align*}
 where $D_\Omega$ is a positive constant defined by
 \begin{align*}
  D_\Omega= \int_\Omega   \frac{\alpha^2\sqrt{\lambda-\alpha}}{4\pi\lambda^3}  d\lambda.
 \end{align*}
  In view of the estimate in Corollary~\ref{corACest2} and the bound in~\eqref{eqnNconvalpha}, we conclude that 
   \begin{align*}
  \mu(\Omega) & \geq  D_\Omega \E^{\frac{-\alpha^2 M}{4 D_\Omega}} > 0.
 \end{align*}
  Since all Borel measures on $\R$ are regular, this readily implies that $\mu(\Omega)$ is positive for every Borel set $\Omega\subseteq[\alpha,\infty)$ of positive Lebesgue measure. 
  Thus, we have finally verified that the essential spectrum of $S$ coincides with the interval $[\alpha,\infty)$ and the absolutely continuous spectrum of $S$ is essentially supported on $[\alpha,\infty)$. 
  
  In order to finish the proof of Theorem~\ref{thmalpha}, let us suppose that $S$ is a generalized indefinite string $(L,\omega,\dip)$ such that  $L$ is infinite and~\eqref{eqnCondSalpha} holds for a real constant $c$ and a positive constant $\eta$.
  We consider the generalized indefinite string $({L},\tilde{\omega},\tilde{\dip})$, where $\tilde{\omega}$ is defined via its normalized anti-derivative $\tilde{\Wr}$ by~\eqref{eqnWrtilde} and $\tilde{\dip} = \eta^{-2} \dip$.
  Since $(L,\tilde{\omega},\tilde{\dip})$ satisfies the assumptions imposed before, we infer that the essential spectrum of $(L,\tilde{\omega},\tilde{\dip})$ coincides with the interval $[\alpha,\infty)$ and its absolutely continuous spectrum is essentially supported on $[\alpha,\infty)$. 
  However, since the corresponding Weyl--Titchmarsh functions $m$ and $\tilde{m}$ are related via~\eqref{eqnmtilde}, we see that the essential spectrum of $S$ coincides with the interval $[\alpha/\eta,\infty)$ and its absolutely continuous spectrum is essentially supported on $[\alpha/\eta,\infty)$.
\end{proof}

 \section{The conservative Camassa--Holm flow}\label{secAPP}

  In this section, we are going to demonstrate how our results apply to the isospectral problem of the conservative Camassa--Holm flow. 
  To this end, let $u$ be a real-valued function in $H^1_\loc[0,\infty)$ and $\dip$ be a non-negative Borel measure on $[0,\infty)$. 
  We define the distribution $\omega$ in $H^{-1}_\loc[0,\infty)$ by  
\begin{align}\label{eqnDefomega}
 \omega(h) = \int_0^\infty u(x)h(x)dx + \int_0^\infty u'(x)h'(x)dx, \quad h\in H^1_\cc[0,\infty),
\end{align}
so that $\omega = u - u''$ in a distributional sense. 
 Now the isospectral problem of the conservative Camassa--Holm flow is associated with the differential equation
 \begin{align}\label{eqnCHISP}
 -g'' + \frac{1}{4} g = z\, \omega\, g + z^2 \dip\, g,
\end{align}
where $z$ is a spectral parameter.
 Just like for generalized indefinite strings, this differential equation has to be understood in a weak sense in general:   
  A solution of~\eqref{eqnCHISP} is a function $g\in H^1_{\loc}[0,\infty)$ such that 
 \begin{align}\label{eqnDEweakform}
   \Delta_g h(0) + \int_{0}^\infty g'(x) h'(x) dx + \frac{1}{4} \int_0^\infty g(x)h(x)dx = z\, \omega(gh) + z^2 \dip(g h) 
 \end{align} 
 for some constant $\Delta_g\in\C$ and every function $h\in H^1_\cc[0,\infty)$.
 For such a solution $g$, the constant $\Delta_g$ is uniquely determined and will be denoted with $g'(0-)$. 
 
 We are first going to show that it is always possible to transform the differential equation~\eqref{eqnCHISP} into the differential equation
    \begin{align}\label{eqntildeString}
    - f'' = z\,\omega_\Sr f + z^2\dip_\Sr f
   \end{align}
   for some corresponding generalized indefinite string $(\infty,\omega_\Sr,\dip_\Sr)$.
 To this end, let us  introduce the diffeomorphism $\Sr\colon[0,\infty)\rightarrow[0,\infty)$ by  
 \begin{align}\label{eq:Sdiff}
  \Sr(t) = \log(1+t), \quad t\in[0,\infty),
 \end{align} 
 and note that the inverse of $\Sr$ is simply given by 
  \begin{align}
  \Sr^{-1}(x) = \E^{x} - 1, \quad x\in[0,\infty).
 \end{align} 
Next we define a real-valued measurable function $\Wr_\Sr$ on $[0,\infty)$ such that   
 \begin{align}\label{eqnDefa}
   \Wr_\Sr(t)  =  u(0) - \frac{u'(\Sr(t))+ u(\Sr(t))}{1+t}  
 \end{align}
 for almost all $t\in[0,\infty)$, where we note that the right-hand side is well-defined almost everywhere.
 It follows readily that the function $\Wr_\Sr$ is locally square integrable, so that we can find a real-valued distribution $\omega_\Sr$ in $H^{-1}_\loc[0,\infty)$ which has $\Wr_\Sr$ as its normalized anti-derivative. 
 Furthermore, the non-negative Borel measure $\dip_\Sr$ on $[0,\infty)$ is given by setting 
 \begin{align}\label{eqnDefbeta}
   \dip_\Sr(B) =  \int_B \frac{1}{1+t}\, d\dip\circ \Sr(t)  = \int_{\Sr(B)} \E^{-x} d\dip(x)
 \end{align}
 for every Borel set $B\subseteq[0,\infty)$.
 This defines a generalized indefinite string $(\infty,\omega_\Sr,\dip_\Sr)$ whose relation to the differential equation~\eqref{eqnCHISP} we are going to describe now. 

 \begin{lemma}\label{lem:sol=sol}
   A function $g$ is a solution of the differential equation~\eqref{eqnCHISP} if and only if the function $f$ defined by 
   \begin{align}\label{eqnfg}
     f(t) =   g(\Sr(t)) \sqrt{1+t}, \quad t\in[0,\infty),
   \end{align}
   is a solution of the differential equation~\eqref{eqntildeString}. 
 \end{lemma}

 \begin{proof}
   Let us suppose that two functions $f$ and $g$ on $[0,\infty)$ are related via~\eqref{eqnfg}. 
   We first note that $f$ belongs to $H^1_\loc[0,\infty)$ if and only if so does $g$. 
   In this case, for a given function $h_\Sr$ in $H^1_\cc[0,\infty)$, a substitution yields 
   \begin{align*}
     \int_0^\infty f'(t)h_\Sr'(t)dt & = \int_0^\infty \biggl(g'(\Sr(t))+\frac{g(\Sr(t))}{2}\biggr) \biggl(h'(\Sr(t)) + \frac{h(\Sr(t))}{2}\biggr) \Sr'(t) dt \\
      & = \int_0^\infty g'(x)h'(x)dx + \frac{1}{4} \int_0^\infty g(x)h(x)dx - \frac{1}{2} g(0)h(0),
   \end{align*} 
   where the functions $h$ and $h_\Sr$ in $H^1_\cc[0,\infty)$ are related by 
   \begin{align*}
     h(x) &=  h_\Sr(\Sr^{-1}(x))\E^{-\frac{x}{2}}, \quad x\in[0,\infty), &
     h_\Sr(t) & = h(\Sr(t))\sqrt{1+t}, \quad t\in[0,\infty).
    \end{align*}
    Furthermore, one computes that
   \begin{align*}
     \int_0^\infty \Wr_\Sr(t) (fh_\Sr)'(t)dt & = - \int_0^\infty u(x)g(x) h(x)dx - \int_0^\infty u'(x)(gh)'(x)dx
   \end{align*} 
   as well as
   \begin{align*}
     \int_{[0,\infty)} fh_\Sr\, d\dip_\Sr & = \int_{[0,\infty)} g(\Sr(t))h(\Sr(t)) d\dip\circ\Sr(t) =  \int_{[0,\infty)} gh\, d\dip.
   \end{align*}
   With the help of these identities, the claim follows readily from the very definition of the respective solutions.
 \end{proof}
 
  In conjunction with \cite[Lemma~4.2]{IndefiniteString} and a simple substitution, this relation readily provides the following result. 
 
 \begin{corollary}\label{cor:psi=psi}
   If $z$ belongs to $\C\backslash\R$, then there is an (up to scalar multiples) unique non-trivial solution $\psi$ of the differential equation~\eqref{eqnCHISP} such that $\psi$ lies in $H^1[0,\infty)$ and $L^2([0,\infty);\dip)$. 
 \end{corollary}
 
 \begin{proof}
  Let us suppose that two functions $f$ and $g$ in $H^1_\loc[0,\infty)$ are related via~\eqref{eqnfg}.
  Then the function $f$ lies in  $\dot{H}^1[0,\infty)$ if and only if the function $g$ lies in $H^1[0,\infty)$.
  In fact, if $f$ lies in $\dot{H}^1[0,\infty)$, then a substitution shows that 
  \begin{align*}
    \int_0^\infty |f'(t)|^2 dt  = \int_0^\infty \Bigl| g'(x) + \frac{1}{2}g(x)\Bigr|^2 dx < \infty. 
  \end{align*}
  Now upon noting that for $R>0$ we have 
  \begin{align*}
    \int_0^R \Bigl| g'(x) + \frac{1}{2}g(x)\Bigr|^2 dx = \int_0^R |g'(x)|^2 dx + \frac{1}{4} \int_0^R |g(x)|^2 dx +  \frac{|g(R)|^2 - |g(0)|^2}{2}, 
  \end{align*}
  we see that $g$ lies in $H^1[0,\infty)$ since it is bounded, which follows because $f$ grows at most like a square root. 
  The converse implication is straightforward. 
  Moreover, we easily see that the function $f$ lies in $L^2([0,\infty);\dip_\Sr)$ if and only if the function $g$ lies in $L^2([0,\infty);\dip)$.
   Now the claim follows readily from~\cite[Lemma~4.2]{IndefiniteString}. 
 \end{proof}

 This result allows us to define the Weyl--Titchmarsh function $m$ associated with the spectral problem~\eqref{eqnCHISP}  by 
 \begin{align}
 m(z) =  \frac{\psi'\NLz}{z\psi(z,0)},\quad z\in\C\backslash\R,
 \end{align} 
 where $\psi(z,\redot)$ is a non-trivial solution of the differential equation~\eqref{eqnCHISP} which lies in $H^1[0,\infty)$ and $L^2([0,\infty);\dip)$. 
 In view of Lemma \ref{lem:sol=sol}, we readily compute that 
 \begin{align}\label{eqnmms}
   m(z) = m_\Sr(z) - \frac{1}{2z}, \quad z\in\C\backslash\R, 
 \end{align}
where $m_\Sr$ is the Weyl--Titchmarsh function of the corresponding generalized indefinite string $(\infty,\omega_\Sr,\dip_\Sr)$.  
 In particular, we see that $m$ is a Herglotz--Nevanlinna function and the Borel measure $\mu$ in the corresponding integral representation differs from the one for the generalized indefinite string only by a point mass at zero. 
 The measure $\mu$ is a spectral measure for a suitable self-adjoint realization $\T$ of the spectral problem~\eqref{eqnCHISP}; compare \cite{bebrwe08, LeftDefiniteSL, CHPencil}. 
 Since this establishes an immediate connection between the spectral properties of $\T$ and the corresponding generalized indefinite string $(\infty,\omega_\Sr,\dip_\Sr)$, we may now invoke Theorem~\ref{thmalpha}. 
 
 \begin{theorem}\label{thmApp}
  If the function $u-1$ belongs to $H^1[0,\infty)$ and the measure $\dip$ is finite, then the essential spectrum of $\T$ coincides with the interval $[1/4,\infty)$ and the absolutely continuous spectrum of $\T$ is essentially supported on $[1/4,\infty)$. 
 \end{theorem}
 
 \begin{proof} 
  Under these assumptions, we readily see that the coefficients of the corresponding generalized indefinite string $(\infty,\omega_\Sr,\dip_\Sr)$ satisfy 
 \begin{align*}
  \int_0^\infty \Bigl|\Wr_\Sr(t) + 1 -u(0) - \frac{t}{1+t} \Bigr|^2 t\, dt & \leq \int_0^\infty |u'(\Sr(t))+ u(\Sr(t))-1|^2 \frac{1}{1+t} dt  \\
     & = \int_0^\infty |u'(x)+ u(x)-1|^2 dx < \infty,
  \end{align*}
  upon performing a substitution, as well as 
  \begin{align*}
   \int_{[0,\infty)} t\, d\dip_\Sr(t) \leq \int_{[0,\infty)} d\dip\circ\Sr = \int_{[0,\infty)} d\dip <\infty. 
 \end{align*}
 Now the claim follows from Theorem~\ref{thmalpha} with $c=u(0)-1$, $\alpha=1/4$ and $\eta=1$. 
 \end{proof}

Of course, it is also desirable to consider the spectral problem for~\eqref{eqnCHISP} on the whole real line. 
In this case, one can show,  using a standard argument based on stability of the absolutely continuous spectrum under finite rank perturbations, that the essential spectrum coincides with the interval $[1/4,\infty)$ and the absolutely continuous spectrum is of multiplicity two and essentially supported on $[1/4,\infty)$ if $u$ is a real-valued function on $\R$ such that $u-1$ belongs to $H^1(\R)$ and $\dip$ is a non-negative finite Borel measure on $\R$;  see Theorem~\ref{thmCHac}.

\section{Schr\"odinger operators with  \texorpdfstring{$\delta'$}{delta-prime}-interactions}

Our main theorems also apply to Schr\"odinger operators with $\delta'$-interactions. 
To this end, let $\nu$ be a real-valued Borel measure on $[0,\infty)$ which is singular with respect to the Lebesgue measure. 
For the sake of simplicity, we shall also assume that $\nu$ does not have a point mass at zero. 
The Borel measure $\omega$ on $[0,\infty)$ is then defined as the sum of the measure $\nu$ and the Lebesgue measure, that is, 
\begin{align}\label{eq:omeganu}
\omega(B) = \nu(B) + \int_Bdx 
\end{align}
for every Borel set $B\subseteq [0,\infty)$. 
We consider the operator $H_\nu$ in the Hilbert space $L^2[0,\infty)$ associated with the differential expression
\begin{align}\label{eq:taunu}
\tau_\nu = -\frac{d}{dx}\frac{d}{d\omega(x)}
\end{align}
and subject to Neumann boundary conditions at zero. 
The operator $H_\nu$ can be viewed as a  {\em Hamiltonian with $\delta'$-interactions}. 
Namely, if $\nu$ is a discrete measure such that 
\begin{align}
\nu = \sum_{s\in X}\beta(s) \delta_s,
\end{align}
where $X$ is a discrete subset of $[0,\infty)$, $\beta$ is a real-valued function on $X$ and $\delta_s$ is the unit Dirac measure centred at $s$, then the differential expression $\tau_\nu$ can be formally written as (see \cite[Example 2.2]{dprime})
\begin{align}
 -\frac{d^2}{dx^2} + \sum_{s\in X}\beta(s)\langle \ledot,\delta_s'\rangle\delta_s',
\end{align}
which is the Hamiltonian with $\delta'$-interactions on $X$ of strength $\beta$ (see \cite{AGHH,km13,km14}).
It is known (see~\cite{dprime} and~\cite{MeasureSL}) that under the above assumption on $\nu$, the operator $H_\nu$ is self-adjoint in $L^2[0,\infty)$. 
The spectral properties of $H_\nu$ turn out to be closely connected with those of the generalized indefinite string $S_\nu = (\infty,\omega,0)$.

\begin{lemma}\label{lemUnitEqv}
The operator $H_\nu$ is unitarily equivalent to $S_\nu$.
\end{lemma} 

\begin{proof}
Since $\tau_\nu$ is in the limit point case at $\infty$ (see \cite{dprime}), for every $z\in\C\backslash\R$ there is an (up to scalar multiples) unique non-trivial solution $\psi_\nu(z,\redot)$ to $\tau_\nu y = zy$ such that $\psi_\nu(z,\redot)$ lies in $L^2[0,\infty)$. 
Recall that the Weyl--Titchmarsh function $m_\nu$ of $H_\nu$ is then given by
\begin{align*}
m_\nu(z) = - \frac{\psi_\nu(z,0)}{\psi_\nu^\qd(z,0)},\quad z\in \C\backslash\R,
\end{align*}
where the function in the denominator is the quasi-derivative 
\begin{align*}
\psi_\nu^\qd(z,\redot) = \frac{d\psi_\nu(z,\redot)}{d\omega}.
\end{align*}
It is known that $m_\nu$ is a Herglotz--Nevanlinna function and the Borel measure in the corresponding integral representation is a spectral measure for the operator $H_\nu$. 
Now for every function $h\in H^1_\cc[0,\infty)$ we compute 
\begin{align*}
  \int_0^\infty \psi_\nu^{\qd\prime}(z,x)h'(x)dx & = z \psi_\nu(z,0) h(0) + z\int_0^\infty \psi_\nu^\qd(z,x) h(x) d\omega(x),
\end{align*}
where we used the fact that $\psi_\nu^{\qd\prime} = -z\psi_\nu$ as well as an integration by parts.
This shows that the quasi-derivative is a solution of the differential equation
  \begin{align*}
  -f'' = z\, \omega f
 \end{align*}
with $\psi_\nu^{\qd\prime}(z,0-)=-z\psi_\nu(z,0)$. 
Since it furthermore lies in $\dot{H}^1[0,\infty)$, the Weyl--Titchmarsh function $m$ of the generalized indefinite string $S_\nu$ is given by 
\begin{align*}
  m(z) = \frac{\psi_\nu^{\qd\prime}(z,0-)}{z\psi_\nu^\qd(z,0)} = \frac{-z\psi_\nu(z,0)}{z\psi_\nu^\qd(z,0)} = m_\nu(z), \quad z\in\C\backslash\R.
\end{align*}
Thus, the corresponding spectral measures coincide and hence the claim follows.
\end{proof}

The established connection now allows us to apply Theorem~\ref{thm0}. 

\begin{theorem}\label{th:ACprime}
Suppose that 
\begin{align}\label{eq:nu0}
\int_{0}^\infty |\Vr(x) - c|^2dx<\infty
\end{align}
for a real constant $c$, where $\Vr$ is the normalized anti-derivative of $\nu$. 
Then the essential spectrum of the Hamiltonian $H_\nu$ coincides with the interval $[0,\infty)$ and the absolutely continuous spectrum of $H_\nu$ is essentially supported on $[0,\infty)$.
\end{theorem}

\begin{proof}
This is a simple consequence of Lemma~\ref{lemUnitEqv} and Theorem~\ref{thm0}.
\end{proof}

\begin{remark}
In conclusion, let us mention that one can say more about spectral properties of the Hamiltonian $H_\nu$ and about its negative spectrum in particular. 
 Namely, it is possible to show that the negative spectrum consists of simple eigenvalues which may only accumulating at $-\infty$. 
 More specifically, they satisfy the following Lieb--Thirring-type bound
\begin{align}\label{eq:LTdelta}
\sum_{\lambda \in \sigma(H_\nu)\cap(-\infty,0)} \frac{1}{|\lambda|^{3/2}} \le \frac{3}{4}\int_{0}^\infty |\Vr(x) - \Vr_0|^2dx.
\end{align}
However, we postpone the discussion of Lieb--Thirring-type inequalities to a forthcoming publication.
\end{remark}

\bigskip
\noindent
\section*{Acknowledgments}
We gratefully acknowledge the kind hospitality at the State Key Laboratory of Scientific and Engineering Computing (LSEC), Academy of Mathematics and Systems Science, Chinese Academy of Sciences in Beijing, China during a stay in July 2018, where a part of this work was done.


\begin{thebibliography}{XX}

\bibitem{AGHH}
S.\ Albeverio, F.\ Gesztesy, R.\ Hoegh-Krohn, H.\ Holden, {\em Solvable Models in Quantum Mechanics}, 2nd ed., AMS Chelsea Publishing, Providence, RI, 2005.

\bibitem{ba01}
H.\ Bauer, {\em Measure and integration theory}, De Gruyter Studies in Mathematics, 26, Walter de Gruyter \& Co., Berlin, 2001.

\bibitem{bebrwe08}
C.\ Bennewitz, B.\ M.\ Brown and R.\ Weikard, {\em Inverse spectral and scattering theory for the half-line left-definite Sturm--Liouville problem}, SIAM J.\ Math.\ Anal.\ {\bf 40} (2008/09), no.~5, 2105--2131.

\bibitem{bebrwe12}
C.\ Bennewitz, B.\ M.\ Brown and R.\ Weikard, {\em Scattering and inverse scattering for a left-definite Sturm--Liouville problem}, J.\ Differential Equations {\bf 253} (2012), no.~8, 2380--2419.

\bibitem{bd17}
R.\ V.\ Bessonov and S.\ A.\ Denisov, {\em A spectral Szeg\H{o} theorem on the real line}, preprint, \arxiv{1711.05671}.

\bibitem{bo07}
V.\ I.\ Bogachev, {\em Measure theory}, Vol.\ I, II, Springer-Verlag, Berlin, 2007. 

\bibitem{bf60}
V.\ S.\ Buslaev and L.\ D.\ Faddeev, {\em On formulas for traces of a Sturm-Liouville singular differential operator}, Dokl.\ Akad.\ Nauk SSSR {\bf 132} (1960), no.~1, 13--16.

\bibitem{brco07}
A.\ Bressan and A.\ Constantin, {\em Global conservative solutions of the Camassa--Holm equation}, Arch.\ Ration.\ Mech.\ Anal.\ {\bf 183} (2007), no.~2, 215--239.

\bibitem{caho93}
R.\ Camassa and D.\ Holm, {\em An integrable shallow water equation with peaked solitons}, Phys.\ Rev.\ Lett.\ {\bf 71} (1993), no.~11, 1661--1664.

\bibitem{chlizh06}
M.\ Chen, S.\ Liu and Y.\ Zhang, {\em A two-component generalization of the Camassa--Holm equation and its solutions}, Lett.\ Math.\ Phys.\ {\bf 75} (2006), no.~1, 1--15. 

\bibitem{co01}
A.\ Constantin, {\em On the scattering problem for the Camassa--Holm equation}, R.\ Soc.\ Lond.\ Proc.\ Ser.\ A Math.\ Phys.\ Eng.\ Sci.\ {\bf 457} (2001), no.~2008, 953--970.

\bibitem{coiv06} 
A.\ Constantin and R.\ I.\ Ivanov, {\em Poisson structure and action-angle variables for the Camassa--Holm equation}, Lett.\ Math.\ Phys.\ {\bf 76},   (2006) 93--108.

\bibitem{coiv08}
A.\ Constantin and R.\ I.\ Ivanov, {\em On an integrable two-component Camassa--Holm shallow water system}, Phys.\ Lett.\ A {\bf 372} (2008), no.~48, 7129--7132. 

\bibitem{deki99}
P.\ Deift and R.\ Killip, {\em On the absolutely continuous spectrum of one-dimensional Schr\"odinger operators with square summable potentials}, Comm.\ Math.\ Phys.\ {\bf 203} (1999), no.~2, 341--347.

\bibitem{dk05}
S.\ A.\ Denisov and A.\ Kiselev, {\em Spectral properties of Schr\"odinger operators with decaying potentials}, in: F.\ Gesztesy et.\ al., ``Spectral Theory and Mathematical Physics: A Festschrift in Honor of Barry Simon's 60th Birthday. Part 2", Proc.\ Symp.\ Pure Math.\ {\bf 76}, pp. 565--589, (2007).

\bibitem{LeftDefiniteSL}
J.\ Eckhardt, {\em Direct and inverse spectral theory of singular left-definite Sturm--Liouville operators}, J.\ Differential Equations {\bf 253} (2012), no.~2, 604--634.

\bibitem{ConservCH}
J.\ Eckhardt, {\em The inverse spectral transform for the conservative Camassa--Holm flow with decaying initial data}, Arch.\ Ration.\ Mech.\ Anal.\ {\bf 224} (2017), no.~1, 21--52.

\bibitem{LagrangeCH}
J.\ Eckhardt and K.\ Grunert, {\em A Lagrangian view on complete integrability of the two-component Camassa--Holm system}, J.\ Integrable Syst.\ {\bf 2} (2017), no.~1, xyx002, 14 pp.

\bibitem{ConservMP}
J.\ Eckhardt and A.\ Kostenko, {\em An isospectral problem for global conservative multi-peakon solutions of the Camassa--Holm equation}, Comm.\ Math.\ Phys.\ {\bf 329} (2014), no.~3, 893--918.

\bibitem{IndefiniteString}
J.\ Eckhardt and A.\ Kostenko, {\em The inverse spectral problem for indefinite strings}, Invent.\ Math.\ {\bf 204} (2016), no.~3, 939--977. 

\bibitem{CHPencil}
J.\ Eckhardt and A.\ Kostenko, {\em Quadratic operator pencils associated with the conservative Camassa--Holm flow}, Bull.\ Soc.\ Math.\ France {\bf 145} (2017), no.~1, 47--95.

\bibitem{IndMoment}
J.\ Eckhardt and A.\ Kostenko, {\em The classical moment problem and generalized indefinite string},  Integr.\ Equat.\ Oper.\ Theory {\bf 90}, 2:23 (2018).

\bibitem{dprime}
J.\ Eckhardt, A.\ Kostenko, M.\ Malamud, and G.\ Teschl, {\em One-dimensional Schr\"odinger operators with $\delta'$-interactions on Cantor-type sets},  J.\ Differential Equations {\bf 257} (2014), 415--449.

\bibitem{MeasureSL}
J.\ Eckhardt and G.\ Teschl, {\em Sturm--Liouville operators with measure-valued coefficients}, J.\ Anal.\ Math.\ {\bf 120} (2013), no.~1, 151--224.

\bibitem{faza71}
L.\ D.\ Faddeev and V.\ E.\ Zakharov, {\em Korteweg--de Vries equation: A completely integrable Hamiltonian system}, 
Funktsional.\ Anal.\ i Prilozhen.\ {\bf 5}, no.~4 (1971), 18--27; {\em English transl.}: Funct.\ Anal.\ Appl.\ {\bf 5}, no.~4, 280--287  (1971).

 \bibitem{fl96}
A.\ Fleige, {\em Spectral theory of indefinite Krein--Feller differential operators}, Mathematical Research, 98, Akademie Verlag, Berlin, 1996.

\bibitem{grhora12}
K.\ Grunert, H.\ Holden and X.\ Raynaud, {\em Global solutions for the two-component Camassa--Holm system}, Comm.\ Partial Differential Equations {\bf 37} (2012), no.~12, 2245--2271.

\bibitem{hest65}
E.\ Hewitt and K.\ Stromberg, {\em Real and Abstract Analysis}, Springer, New York, 1965.

\bibitem{hora07}
H.\ Holden and X.\ Raynaud, {\em Global conservative solutions of the Camassa--Holm equation---a Lagrangian point of view}, Comm.\ Partial Differential Equations {\bf 32} (2007), no.~10-12, 1511--1549.

\bibitem{hoiv11}
D.\ D.\ Holm and R.\ I.\ Ivanov, {\em Two-component CH system: inverse scattering, peakons and geometry}, Inverse Problems {\bf 27} (2011), no.~4, 045013, 19 pp.

\bibitem{hs14}
D.\ Hughes and K.\ M.\ Schmidt, {\em Absolutely continuous spectrum of Dirac operators with square-integrable potentials}, Proc.\ Roy.\ Soc.\ Edinburgh Sect.\ A {\bf 144} (2014), no.~3, 533--555.

\bibitem{ki02}
R.\ Killip, {\em Perturbations of one-dimensional Schr\"odinger operators preserving the absolutely continuous spectrum},
Intern.\ Math.\ Res.\ Notices IMRN {\bf 2002}, no.~38, (2002) 2029--2061.  

\bibitem{ko13}
 A.\ Kostenko, {\em The similarity problem for indefinite Sturm--Liouville operators and the HELP inequality}, Adv.\ Math.\ {\bf 246} (2013), 368--413.

\bibitem{km13}
A.\ Kostenko and M.\ Malamud, {\em 1-D Schr\"odinger operators with local point interactions: a review}, in: H.\ Holden, et al.
(Eds.), ``Spectral Analysis, Integrable Systems, and Ordinary Differential Equations", in: Proc. Sympos. Pure Math.,
{\bf 87}, Amer. Math. Soc., Providence, 2013, pp. 235--262.

\bibitem{km14}
A.\ Kostenko and M.\ Malamud, {\em Spectral theory of semibounded Schr\"odinger operators with $\delta'$-interactions}, 
Ann.\ Henri Poincar\'e {\bf 15}, no.~3, (2014) 501--541.

\bibitem{mnv01}
S.\ Molchanov, M.\ Novitskii, and B.\ Vainberg, {\em First KdV integrals and absolutely continuous spectrum for 1-D Schr\"odinger operator}, Comm.\ Math.\ Phys.\ {\bf 216} (2001), no.~1, 195--213.

\bibitem{roro94}
M.\ Rosenblum and J.\ Rovnyak, {\em Topics in Hardy classes and univalent functions}, Birkh\"{a}user Verlag, Basel, 1994.

\bibitem{ru74}
W.\ Rudin, {\em Real and complex analysis}, third edition, McGraw-Hill Book Co., New York, 1987. 

\bibitem{ry05}
A.\ Rybkin, {\em On the spectral $L_2$ conjecture, $3/2$-Lieb--Thirring inequality and distributional potentials}, J.\ Math.\ Phys.\ {\bf 46} (2005), no.~12, 123505, 8~pp.

\bibitem{sim00}
B.\  Simon, {\em Schr\"odinger  operators  in  the  twenty-first  century},  in: ''Mathematical Physics  2000",  A.~Fokas,  A.~Grigoryan,  T.~Kibble  and  B.~Zegarlinski (eds.), 283--288, Imperial College Press, London, 2001.

\bibitem{sim18}
B.\ Simon, {\em Tosio Kato's work on non-relativistic quantum mechanics: part 1}, Bull.\ Math.\ Sci.\ {\bf 8} (2018), no.~1, 121--232.

\bibitem{sim17}
B.\ Simon, {\em Tosio Kato's work on non-relativistic quantum mechanics: part 2}, Bull.\ Math.\ Sci.\ (to appear), \doi{10.1007/s13373-018-0121-5}.

\end{thebibliography}
\end{document}